\documentclass[reqno]{amsart}

\usepackage[all]{xy}
\usepackage{graphicx}

\usepackage{amssymb}
\usepackage{amsmath}
\usepackage{mathrsfs}
\usepackage{epsfig}
\usepackage{amscd}
\usepackage{graphicx, color}

\def\E{\ifmmode{\mathbb E}\else{$\mathbb E$}\fi} 
\def\N{\ifmmode{\mathbb N}\else{$\mathbb N$}\fi} 
\def\R{\ifmmode{\mathbb R}\else{$\mathbb R$}\fi} 
\def\Q{\ifmmode{\mathbb Q}\else{$\mathbb Q$}\fi} 
\def\C{\ifmmode{\mathbb C}\else{$\mathbb C$}\fi} 
\def\H{\ifmmode{\mathbb H}\else{$\mathbb H$}\fi} 
\def\Z{\ifmmode{\mathbb Z}\else{$\mathbb Z$}\fi} 
\def\P{\ifmmode{\mathbb P}\else{$\mathbb P$}\fi} 
\def\T{\ifmmode{\mathbb T}\else{$\mathbb T$}\fi} 
\def\SS{\ifmmode{\mathbb S}\else{$\mathbb S$}\fi} 
\def\DD{\ifmmode{\mathbb D}\else{$\mathbb D$}\fi} 

\newcommand{\del}{\partial}
\newcommand{\Cont}{{\operatorname{Cont}}}

\newcommand{\Fred}{{\operatorname{Fred}}}

\newcommand{\ben}{\begin{enumerate}}
\newcommand{\een}{\end{enumerate}}
\newcommand{\be}{\begin{equation}}
\newcommand{\ee}{\end{equation}}
\newcommand{\bea}{\begin{eqnarray}}
\newcommand{\eea}{\end{eqnarray}}
\newcommand{\beastar}{\begin{eqnarray*}}
\newcommand{\eeastar}{\end{eqnarray*}}
\newcommand{\bc}{\begin{center}}
\newcommand{\ec}{\end{center}}

\theoremstyle{theorem}
\newtheorem{thm}{Theorem}[section]
\newtheorem{cor}[thm]{Corollary}
\newtheorem{lem}[thm]{Lemma}
\newtheorem{prop}[thm]{Proposition}

\theoremstyle{definition}
\newtheorem{defn}[thm]{Definition}
\newtheorem{rem}[thm]{Remark}

\newtheorem{hypo}[thm]{Hypothesis}

\newtheorem{notation}[thm]{\rm\bfseries{Notation}}

\newtheorem*{thm*}{Theorem}

\numberwithin{equation}{section}

\def\R{{\mathbb R}}

\def\Crit{{\hbox{Crit}}}

\def\E{{\mathbb E}}
\def\Z{{\mathbb Z}}
\def\C{{\mathbb C}}
\def\R{{\mathbb R}}
\def\P{{\mathbb P}}

\def\N{{\mathbb N}}

\def\11{{\mathbb I}}

\def\delbar{{\overline \partial}}

\def\C{\mathbb{C}}
\def\Z{\mathbb{Z}}

\def\T{\mathbb{T}}

\def\Q{\mathbb{Q}}

\def\E{\ifmmode{\mathbb E}\else{$\mathbb E$}\fi} 
\def\N{\ifmmode{\mathbb N}\else{$\mathbb N$}\fi} 
\def\R{\ifmmode{\mathbb R}\else{$\mathbb R$}\fi} 
\def\Q{\ifmmode{\mathbb Q}\else{$\mathbb Q$}\fi} 
\def\C{\ifmmode{\mathbb C}\else{$\mathbb C$}\fi} 
\def\H{\ifmmode{\mathbb H}\else{$\mathbb H$}\fi} 
\def\Z{\ifmmode{\mathbb Z}\else{$\mathbb Z$}\fi} 
\def\P{\ifmmode{\mathbb P}\else{$\mathbb P$}\fi} 
\def\SS{\ifmmode{\mathbb S}\else{$\mathbb S$}\fi} 
\def\DD{\ifmmode{\mathbb D}\else{$\mathbb D$}\fi} 

\def\R{{\mathbb R}}

\def\Crit{{\hbox{Crit}}}
\def\E{{\mathbb E}}
\def\Z{{\mathbb Z}}
\def\C{{\mathbb C}}
\def\R{{\mathbb R}}

\def\N{{\mathbb N}}
\def\LL{{\mathcal L}}
\def\CC{{\mathcal C}}

\def\delbar{{\overline \partial}}

\def\CB{{\mathcal B}}
\def\CC{{\mathcal C}}

\def\CH{{\mathcal H}}

\def\CJ{{\mathcal J}}

\def\CL{{\mathcal L}}
\def\CM{{\mathcal M}}

\def\CV{{\mathcal V}}
\def\CW{{\mathcal W}}

\def\darr#1{\raise1.5ex\hbox{$\leftrightarrow$}
\mkern-16.5mu #1}

\def\roughly#1{\raise.3ex\hbox{$#1$\kern-.75em
\lower1ex\hbox{$\sim$}}}

\def\opname#1{\mathop{\kern0pt{\rm #1}}\nolimits}

\def\Im{\opname{Im}}
\def\End{\opname{End}}
\def\dim{\opname{dim}}

\def\End{\operatorname{End}}

\def\span{\operatorname{span}}

\def\Cont{\operatorname{Cont}}
\def\Crit{\operatorname{Crit}}
\def\Spec{\operatorname{Spec}}
\def\Sing{\operatorname{Sing}}
\def\GFQI{\mathfrak{G}}
\def\Index{\operatorname{Index}}

\def\Image{\operatorname{Image}}

\def\Int{\operatorname{Int}}

\begin{document}

\quad \vskip1.375truein

\def\mq{\mathfrak{q}}
\def\mp{\mathfrak{p}}
\def\mH{\mathfrak{H}}
\def\mh{\mathfrak{h}}
\def\ma{\mathfrak{a}}
\def\ms{\mathfrak{s}}
\def\mm{\mathfrak{m}}
\def\mn{\mathfrak{n}}
\def\mz{\mathfrak{z}}
\def\mw{\mathfrak{w}}
\def\Hoch{{\tt Hoch}}
\def\mt{\mathfrak{t}}
\def\ml{\mathfrak{l}}
\def\mT{\mathfrak{T}}
\def\mL{\mathfrak{L}}
\def\mg{\mathfrak{g}}
\def\md{\mathfrak{d}}
\def\mr{\mathfrak{r}}
\def\Cont{\operatorname{Cont}}
\def\Crit{\operatorname{Crit}}
\def\Spec{\operatorname{Spec}}
\def\Sing{\operatorname{Sing}}
\def\GFQI{\text{\rm GFQI}}
\def\Index{\operatorname{Index}}
\def\Cross{\operatorname{Cross}}
\def\Ham{\operatorname{Ham}}
\def\Fix{\operatorname{Fix}}
\def\Graph{\operatorname{Graph}}
\def\ric{\operatorname{Ric}}
\def\lcs{\text{\rm lcs}}
\def\Sym{\text{\rm Sym}}
\def\Char{\text{\rm Char}}

\title[Asymptotic operators]
{Perturbation theory of asymptotic operators of contact instantons and
pseudoholomorphic curves on symplectization}
\author{Taesu Kim}
\address{Department of Mathematics,  
 POSTECH,  77 Cheongam-ro, Nam-gu, Pohang-si, Gyeongbuk, Korea 37673}
 \email{kimtaesu@postech.ac.kr}
\author{Yong-Geun Oh}
\address{Center for Geometry and Physics,  Institute for Basic Science (IBS), 
79 Jigok-ro 127beon-gil, Nam-gu, Pohang, Gyeongbuk, KOREA 37673 \&
Department of Mathematics, POSTECH,
77 Cheongam-ro, Nam-gu, Pohang-si, Gyeongbuk, Korea 37673}
\email{yongoh1@postech.ac.kr}

\thanks{This work is supported by the IBS project \# IBS-R003-D1.
Both authors would like to also acknowledge MATRIX and the Simons Foundation for their support and funding through the MATRIX-Simons Collaborative Fund of the IBS-CGP and MATRIX workshop on Symplectic Topology.
}


\begin{abstract} In this paper, we first derive covariant tensorial formula
for the asymptotic operators of contact instantons $w:\dot \Sigma \to Q$
and of pseudoholomorphic curves
$u : (\dot \Sigma,j) \to (Q \times \R, \widetilde J)$
on the symplectization of contact manifold $(Q,\lambda)$. The formula 
is independent of the choice of connection \emph{on the nose} and
exhibits explicit dependence on the compatible pair $(\lambda,J)$ of
given contact triad $(Q,\lambda, J)$. Then based on
this, we prove that all eigenvalues of the asymptotic operator 
are simple \emph{for a generic choice of compatible CR almost complex structure $J$} 
for given contact form $\lambda$. This perturbation theory has been missing in 
the study of pseudoholomorphic curves on symplectization.
\end{abstract}


\maketitle

\tableofcontents

\section{Introduction and overview}

Let $(Q, \xi)$ be a contact manifold. Assume $\xi$ is coorientable.
Then we can choose a contact form $\lambda$ with $\ker \lambda = \xi$.
With $\lambda$ given, we have the Reeb vector field $R_\lambda$
uniquely determined by the equation $R_\lambda \rfloor d\lambda = 0, \, R_\lambda \rfloor \lambda = 1$.
Then we have decomposition $TQ = \xi \oplus \R \{R_\lambda\}$. We denote by $\pi: TQ \to \xi$ the associated projection and
$\Pi = \Pi_\lambda: TQ \to TQ$ the associated idempotent whose image
is $\xi$.

A \emph{contact triad} is a triple $(Q,\lambda, J)$
where $\lambda$ is a contact form of $\xi$, i.e., $\ker \lambda = \xi$ and
$J$ is an endomorphism $J: TQ \to TQ$ that satisfies the following.

\begin{defn}[CR almost complex structure]
\label{defn:adapted-J} Let $(Q,\xi = \ker \lambda)$ be as above.
A \emph{CR almost complex structure $J$} is an endomorphism
$J: TQ \to TQ$ satisfying $J^2 = - \Pi$, or more explicitly
$$
(J|_\xi)^2 = - id|_\xi, \quad J(R_\lambda) = 0.
$$
\begin{enumerate}
\item We say $J$ is \emph{adapted to}
$\lambda$ if $d\lambda(Y, J Y) \geq 0$
for all $Y \in \xi$ with equality only when $Y = 0$. In this case, we call the pair $(\lambda, J)$ an \emph{adapted pair} of the contact manifold $(Q,\xi)$.
\item We say $J$ is \emph{compatible to} $\lambda$ if the bilinear form $d\lambda(X,JY)$
is symmetric on $\xi$ in addition to (1). In that case, we call the pair $(\lambda,J)$ a
compatible pair of $(Q,\xi)$.
\end{enumerate}
\end{defn}
The associated contact triad metric for a $\lambda$-compatible pair $(\lambda, J)$ is given by
\be\label{eq:triad-metric}
g_\lambda := d\lambda(\cdot, J\cdot) +\lambda\otimes\lambda.
\ee
The symplectization of $(Q,\lambda)$ is the symplectic manifold $(Q \times \R, d(e^s\lambda))$ with $\R$-coordinate
$s$ also called the \emph{(cylindrical) radial coordinate}.
We equip the symplectization with the $s$-translation invariant almost complex structure
$$
\widetilde J = J \oplus J_0
$$
where $J_0$ is the almost
complex structure on the plane $\R \{\frac{\del}{\del s}, R_\lambda\}$ satisfying $J_0(\frac{\del}{\del s}) = R_\lambda$.

On the other hand, in
the joint work \cite{oh-savelyev} by Savelyev and the second-named 
author, they lifted the theory of contact instantons to
the theory of pseudoholomorphic curves on the $\mathfrak{lcs}$-fication
$(Q \times S^1_\rho, \omega_\lambda)$, Banyaga's locally conformal symplectification (which they call the $\mathfrak{lcs}$-fication) of
contact manifold $(Q,\lambda)$ \cite{banyaga:lcs} on
\be\label{eq:banyaga-lcs}
(Q \times S^1_\rho, d\lambda + d\theta \wedge \lambda)
\ee
with the canonical angular form $d\theta$ satisfying $\int_{S^1_\rho} d\theta = 1$. According to the terminology adopted in \cite{oh-savelyev},
the authors call them the $\mathfrak{lcs}$-fication `of nonzero temperature' 
on which the theory of pseudoholomorphic curves is
developed. Here `$\mathfrak{lcs}$' stands for the standard abbreviation of `locally conformal symplectic'. The authors of ibid. 
call the relevant pseudoholomorphic curves \emph{lcs instantons}.
This family can be augmented by  including the case of the product $Q \times \R$ as the `zero temperature limit' with $1/\rho \to 0$,
\be\label{eq:lcsfication}
(Q \times \R, \omega_\lambda), \quad  \omega_\lambda : = d\pi_Q^*\lambda + ds \otimes \pi_Q^*\lambda
\ee
in physical terms. 
(We refer interested readers to many physical literature for various physical discussion 
involving consideration of such a limit. Here, we just take one such example,  the discussion 
in  \cite[Section 4.1]{FFN-07}.)
Here we denote by $\rho$ the radius of the circle $S^1$ and by $\pi= \pi_Q$ (resp. $s$) the projection to $Q$ (resp. to $\R$) of $Q \times \R$.
For the simplicity of notation, we will often just write $\lambda$ for $\pi_Q^*\lambda$ on $Q \times \R$.

The tensorial approach also clarifies the relationship between
the background geometries of
the contact triad $(Q,\lambda,J)$, the symplectization
$$
(SQ, d(e^s\lambda)) = (Q \times \R, e^s \omega_\lambda)
$$
and the lcs manifold \eqref{eq:banyaga-lcs}.
We write $M: = SQ$ and consider the decomposition
\be\label{eq:TM-splitting}
TM \cong TQ  \oplus \R\cdot \frac{\del}{\del s} \cong \xi \oplus \span\left\{\widetilde R_\lambda,
\frac{\del}{\del s}\right\}  \cong \xi \oplus \R^2.
\ee
We denote by $\widetilde R_\lambda$ the unique vector field on $[0,\infty) \times Q$ which is
invariant under the translation, tangent to the level sets of $s$
and projected to $R_\lambda$. When there is no danger of confusion, we will
mostly denote it by $R_\lambda$.
For give contact triad $(Q,\lambda, J)$, 
we have a canonical almost complex structure
$$
J_0: \R\left\{\frac{\del}{\del s},R_\lambda\right\} \to
\R \left\{\frac{\del}{\del s}, R_\lambda \right\}
$$
defined by $J_0 \frac{\del}{\del s} = R_\lambda$ thanks to the splitting \eqref{eq:TM-splitting}.

Any smooth map
$u:\dot \Sigma \to Q \times \R$ has the form $u = (w,f)$ with
\be\label{eq:a-s}
f = s \circ u, \, w = \pi \circ u
\ee
in the presence of contact form $\lambda$ on $(Q,\xi)$.
We have the decomposition of the derivative
$$
du = dw \oplus \left(df \otimes \frac{\del}{\del s}\right)
$$
viewed as a $TM$-valued one-form 
with respect to the splitting
$$
\operatorname{Hom}(T_z\dot \Sigma, T_{u(z)}M)
= \operatorname{Hom}(T_z\dot \Sigma, HT_{u(z)}M)
\oplus \operatorname{Hom}(T_z\dot \Sigma, VT_{u(z)}M).
$$
(For the notational simplicity, we often omit `$\otimes$' except
the situation that could cause confusion to the readers without it.)

The main purpose of the present paper is to
continue the second-named author and his collaborators'
covariant tensorial study of contact instantons and of the pseudoholomorphic curves on symplectization
given in \cite{oh-wang:CR-map1,oh-wang:CR-map2,oh:contacton-Legendrian-bdy,oh-yso:index}
and to carry out precise asymptotic analyses near the punctures
of finite energy contact instantons and of finite energy pseudoholomorphic curves
by developing a generic perturbation theory
of asymptotic operators over the change of compatible
CR-almost complex structures.

\subsection{Pseudoholomorphic curves in symplectization and contact instantons}

By definition, we have
$
d\pi_Q du = dw.
$
It was observed by Hofer \cite{hofer:invent} that $u$ is $\widetilde J$-holomorphic if and only if $(w,f)$ satisfies
\be\label{eq:tildeJ-holo}
\begin{cases}
\delbar^\pi w = 0 \\
w^*\lambda \circ j= df.
\end{cases}
\ee
(We refer to Appendix \ref{sec:covariant-differential} for the here unexplained notation $\delbar^\pi$.)
A contact instanton is a map $w: \dot \Sigma \to M$ that satisfies
the system of nonlinear partial differential equation
\be\label{eq:contacton-intro}
\delbar^\pi w = 0, \quad d(w^*\lambda \circ j) = 0
\ee
on a contact triad $(M,\lambda, J)$.  The equation itself had been 
introduced by Hofer \cite[p.698]{hofer:gafa}. Note that for any map $u = (w,f)$ satisfying \eqref{eq:tildeJ-holo}
$w$ satisfies this equation with the additional property that the one-form $w^*\lambda \circ j$ is \emph{exact}, not just closed.

In a series of papers, \cite{oh-wang:CR-map1,oh-wang:CR-map2} 
jointed with Wang and in \cite{oh:contacton}, the second-named author systematically
developed analysis of contact instantons (for the closed string case)
\emph{without taking symplectization} by the global covariant tensorial
calculations using the notion of \emph{contact triad connection} 
 introduced by Wang and the second-named author in \cite{oh-wang:connection}.
 A relevant Fredholm theory has been developed by the second-named
 author in \cite{oh:contacton}, \cite{oh-savelyev} (for the closed string case).
More recently the second-named author also studied its open string counterpart of
the boundary value problem of \eqref{eq:contacton-intro} under the Legendrian boundary condition
whose explanation is now in order.  We mention that in early 2000's Abbas  studied the 
Legendrian chord  problem in \cite{abbas:chord} in early 2000's, and more recently
Cant \cite{cant:thesis} developed a detailed Fredholm theory on symplectization
in the relative context.

Throughout the paper, we adopt the following notations.
\begin{notation} We denote by $(\Sigma,j)$ a closed Riemann surface,
$\dot \Sigma$ the associated punctured Riemann surface and
$\overline \Sigma$ the real blow-up of $\dot \Sigma$ along the
punctures.
\end{notation}
For the simplicity and
for the main purpose of the present paper, we will focus on the genus zero case 
so that $\dot \Sigma$ is conformally the unit disc with boundary
punctures $z_0, \ldots, z_k \in \del D^2$ ordered
counterclockwise, i.e.,
$$
\dot \Sigma \cong D^2 \setminus \{z_0, \ldots, z_k\}
$$
Then, for a  $(k+1)$-tuple $\vec R = (R_0,R_1, \cdots, R_k)$ of Legendrian submanifolds, which we call
an (ordered) Legendrian link, we consider
the boundary value problem
\be\label{eq:contacton-Legendrian-bdy-intro}
\begin{cases}
\delbar^\pi w = 0, \quad d(w^*\lambda \circ j) = 0,\\
w(\overline{z_iz_{i+1}}) \subset R_i
\end{cases}
\ee
as an elliptic boundary value problem for a map $w: \dot \Sigma \to M$
and derive the a priori coercive elliptic estimates. Here $\overline{z_iz_{i+1}} \subset \del D^2$
is the open arc between $z_i$ and $z_{i+1}$.

\begin{rem}
Oh also identified the correct 
counterpart of \emph{Hamiltonian-perturbed contact instantons}
 in \cite{oh:contacton-Legendrian-bdy}, \cite{oh:perturbed-contacton-bdy} 
and applied them to a systematic quantitative study of  contact topology
\cite{oh:entanglement1}, and  prove Shelukhin's conjecture  \cite{oh:shelukhin-conjecture}.
\end{rem}

Let $\widetilde \nabla: = \nabla^{\text{\rm can}}$
be the canonical connection of
this almost Hermitian manifold
\be\label{eq:almost-Hermitian-triple}
(Q \times \R, \widetilde g_\lambda, \widetilde J), \quad  \widetilde g_\lambda : = g_\lambda + ds \otimes ds
\ee
i.e., the unique Riemannian connection whose torsion
$T$ satisfies $T(X,\widetilde J X) = 0$ for all $X \in T(Q \times \R)$
(See \cite{gauduchon,kobayashi,oh:book1} for its definition and basic properties.) We note that this almost Hermitian structure on $Q\times \R$
is translational invariant in the radial $s$-direction.
The usage of the canonical connection
on the $\mathfrak{lcs}$-fication $(Q\times \R, \widetilde J, \omega_\lambda)$
(or equivalently via contact triad connection on the triad $(Q,\lambda,J)$
plays an important role in the present authors'
analysis of the \emph{asymptotic operators} of finite energy contact instantons
in Section \ref{sec:spectral-perturbation},
and hence of finite energy pseudoholomorphic curves too. Roughly speaking,
our  coordinate-free approach enables us to compute the asymptotic operator associated to
 each \emph{isospeed Reeb orbit} $(\gamma,T)$ of a contact instanton $w$ denoted by
$$
A^\pi_{(\lambda,J,\nabla)}: \Gamma(\gamma^*\xi) \to \Gamma(\gamma^*\xi)
$$
\emph{simultaneously} over all isospeed Reeb orbits in the covariant tensorial way
\emph{in terms of the pull-back connection under the map $w$} of the
 \emph{given}  contact triad connection 
or the Levi-Civita of the triad metric of the given triad $(Q,\lambda,J)$.
(See Definition \ref{defn:asymptotic-operator} for the precise definition of
the operator $A^\pi_{(\lambda,J,\nabla)}$.)

\begin{defn}[Isospeed Reeb trajectory] We call a pair $(\gamma, T)$ an \emph{isospeed Reeb trajectory}
if $\gamma: [0,1] \to Q$ and $T = \int \gamma^*\lambda$ satisfy
\be\label{eq:isospeed}
\dot \gamma(t) = T R_\lambda(\gamma(t)).
\ee
If $\gamma(1) = \gamma(0)$ in addition, we call it an isospeed Reeb orbit.
\end{defn}

\subsection{Asymptotic operators and their analysis}
\label{subsec:admissiblepair}

We first mention a few differences between the way how we study the asymptotic operators and
those of \cite{HWZ:asymptotics} and of other literature such as \cite[Appendix C]{robbin-salamon:asymptotic}, 
\cite{siefring:relative,siefring:intersection}, \cite{wendl:lecture},
\cite{cant:thesis}.

In \cite[Appendix E]{robbin-salamon:asymptotic},  \cite{siefring:relative,siefring:intersection},
\cite[Section 3.3]{wendl:lecture},
there have been attempts to give a coordinate-free definition
of the asymptotic operator along the
associated asymptotic Reeb orbit for a pseudoholomorphic
curve $u = (w,f)$ on symplectization.
However they fall short of a seamless definition of the
`asymptotic operator' of the Reeb orbits
because the Reeb orbit lives on $Q$ while the pseudoholomorphic curves live on the product $Q \times \R$
and the asymptotic limit of pseudoholomorphic curve live at infinity $Q \times \{\pm \infty\}$.
(Cant studied the asymptotic operator in the relative context in \cite[Section 6.3]{cant:thesis}
by adapting Wendl's.)

 What these literature (e.g. \cite[Section 3.3]{wendl:lecture}, \cite{siefring:relative,siefring:intersection}) 
 are describing, however,
is actually the asymptotic operator of the contact instanton $w$ but trying to describe it in terms of
the pseudoholomorphic curves which prevents them from being able to give
a seamless definition:  Recall the decomposition
\be\label{eq:dtildew}
du = d^\pi w \oplus (w^*\lambda \otimes R_\lambda) \oplus
\left(df \otimes \frac{\del}{\del s}\right)
\ee
with respect to the splitting \eqref{eq:TM-splitting}.
(Compare their practices with our definition of the asymptotic operator of contact instantons given
in Definition \ref{defn:asymptotic-operator} and compare it therewith. 
See also \cite[Section 11.2 \& 11.5]{oh-wang:CR-map2}
for the precursor of our definition.)

In this regard, we will derive an explicit formula of the
asymptotic operator in terms of the contact triad connection,
which exhibits the way how the operator explicitly depends on
$J$ so it admits a perturbation theory under the change of $J$.
\begin{prop}[Proposition \ref{prop:A}]
\label{prop:A-intro}Let $\nabla$ be the triad connection, and $(\gamma, T)$ be
an isospeed Reeb orit. Then
the asymptotic operator $A^\pi_{(\lambda,J,\nabla)}$ is given by
$$
A^\pi_{(\lambda,J,\nabla)} = - J\nabla_t - \frac{T}{2} J \CL_{R_\lambda}J.
$$
\end{prop}

Combining this with the basic properties of the triad connection
and $\dot \gamma(t) = T R_\lambda(\gamma(t))$, we can derive an explicit expression of
the asymptotic operator purely in terms of $\lambda$, $J$ and $(\gamma,T)$
independent of the choice of the connection.

\begin{thm}[Theorem \ref{thm:A=T}]
\label{thm:A=T-intro} Let $(Q,\lambda,J)$ be a triad. Let $(\gamma, T)$ with $T \neq 0$ be an
iso-speed Reeb trajectory. Then we have the formula
\be\label{eq:A-in-CL2-intro}
A^\pi_{(\lambda,J,\nabla)} = T\left(-\frac12 \CL_{R_\lambda}J -  J \CL_{R_\lambda}
- \frac{1}{2} J(\CL_{R_\lambda}J)\right)
\ee
when acted upon $\Gamma(\gamma^*TM)$.
\end{thm}

This study  enables us to make precise statement on the spectral behavior of asymptotic operators under the perturbation of $J$'s.
The ellipticity of $A^\pi_{(\lambda,J,\nabla)}$ implies that it
 has a discrete set of eigenvalues, which we enumerate into
$$
\cdots < \mu_{-k} < \cdots < \mu_{-1} < 0 < \mu_1 < \cdots < \mu_k <
\cdots
$$
repeated with finite multiplicities allowed,
where $\mu_k \to \infty$ (resp. $\mu_{-k} \to -\infty$) as $k \to \infty$.
It also has a uniform spectral gap (\cite[Theorem 6.29]{kato}) in that
there exists some $d> 0$ such that
$|\mu_{i+1} - \mu_i| \geq d > 0$ for all $i$.

Our explicit formula of the asymptotic operator given in Proposition
\ref{prop:A-intro}, which simultaneously applies to all closed Reeb orbits, 
enables us to provide a generic description of the spectral behavior of asymptotic 
operators under the change of $\lambda$-compatible CR almost complex structures $J$ when $\lambda$ is fixed.
For example, we prove
the following natural generic perturbation result of eigenvalues of the asymptotic operator.

\begin{thm}[Generic simpleness of eigenvalues,
Theorem \ref{thm:simple-eigenvalue}]
\label{thm:simple-eigenvalue-intro}  Let $(Q,\xi)$ be a contact manifold.
Assume that $\lambda$ is
nondegenerate. For a generic choice of $\lambda$-compatible CR almost complex
structures $J$, all eigenvalues $\mu_i$ of the asymptotic
operator are simple for all closed Reeb orbits of $\lambda$.
\end{thm}
See Section \ref{sec:spectral-perturbation} for our derivation of the formula of
the asymptotic operator. 

\subsection{Relationship with other literature and discussion}

We have no doubt that a parameterized version of the proof of Theorem \ref{thm:simple-eigenvalue}
will lead to the following spectral flow description of the index. 

Consider the nonlinear Fredholm map 
$$
\Upsilon_{(\lambda,J)}^{\text{\rm cont}}(w): = (\delbar^\pi w, d(w^*\lambda \circ j))
$$
whose domain and codomain are specified as in \cite{oh:contacton} or in \cite{oh-yso:index}
where this map was just written as $\Upsilon$ instead. Since we use the latter for a 
different purpose later in the present article, we use different notation 
$\Upsilon_{(\lambda,J)}^{\text{\rm cont}}$ here for the same
operator appearing in \cite{oh:contacton} or in \cite{oh-yso:index}.

\begin{thm}\label{thm:spectral-flow} Let $(Q,\lambda,J)$ be a contact triad.
Let  $w: \R \times S^1 \to Q$ be a contact instanton with its asymptotic 
limits $(\gamma_+,T_+)$ and $(\gamma_-,T_-)$ and write the associated 
asymptotic operators as 
$$
A^{\pm \infty} = A^{\pi, \pm}_{(\lambda,J,\nabla)}: \Gamma(\gamma_{\pm \infty} ^*\xi) 
\to \Gamma(\gamma_{\pm \infty}^*\xi).
$$
Then for $J$ generic in the sense of Theorem \ref{thm:simple-eigenvalue}, we have
$$
\Index D\Upsilon_{(\lambda,J)}^{\text{\rm cont}}(w) = \mu^{\text{\rm spec}}(A^{- \infty}, A^{+\infty})
$$
where $\mu^{\text{\rm spec}}(A^{- \infty}, A^{+ \infty})$ is the spectral flow from 
$A^{-\infty}$ to $A^{+ \infty}$.
\end{thm}
The main step of the proof is to establish 
the contact instanton version of Lemma 3.17 \cite{wendl:lecture} which reads that there is a
perturbation (inside the space of abstract Fredholm operators) of the given one-parameter family of asymptotic operators 
such that all eigenvalue families of the perturbed operator are transversal in an explicit sense. 
In our case, such a transversality result can be obtained inside the smaller natural family of $J$ perturbations
by the parameterized version of the proof of Theorem \ref{thm:simple-eigenvalue}.
We omit the details of its proof
leaving to interested readers or them to \cite[Appendix C]{wendl:lecture} which handles the
case of pseudoholomorphic curves in symplectization the scheme of which can be now easily 
modified via the perturbations of $J$ utilizing the scheme of
 our proof of Theorem \ref{thm:simple-eigenvalue}.

Because the existing literature on the pseudoholomorphic curves on
symplectization lack an explicit formula of the asymptotic operator as given in
Theorem \ref{thm:A=T-intro}, it has been the case that the general abstract
perturbation theory of linear operators Kato \cite{kato} is just
quoted  in their study of asymptotic operators which prevents one
from making any statement on specific dependence on
the compatible almost complex structures.
(See \cite[Section 3.2]{wendl:lecture}
especially see Lemma 3.17, Theorem 3.35 and Appendix C of
\cite{wendl:lecture} for example which  in turn extends the exposition
given in \cite{HWZ:embedding-control} to higher dimensions.)

As far as we can see, the aforementioned status of matter makes the existing description 
in the literature, at least, of the spectral flow related to the index 
formula of the linearized operator for the $\widetilde J$-pseudoholomorphic curves on 
the fundamental case of 
$\R \times S^1$ for the closed string case (or on $\R \times [0,1]$ for the open string case)
in symplectization or in SFT rather unsatisfying and incomplete.  This is because the spectral flow
definition \emph{over the real line $\R \cong (0,1)$} starts from the requirement that the asymptotic operators 
$$
A^{\pm \infty} = A^{\pi, \pm}_{(\lambda,J,\nabla)}: \Gamma(\gamma_{\pm \infty} ^*\xi) \to \Gamma(\gamma_{\pm \infty}^*\xi)
$$
associated to the given $\widetilde J$ have simple eigenvalues at the end points: 
Note, however, that \emph{these operators may have eigenvalues with multiplicity} before making some
perturbation. (Such a requirement is vacuous over the closed circle $S^1$
considered as in the original article \cite{atiyah-patodi-singer:spectral}.)
Furthermore one needs this description
 \emph{simultaneously over all} asymptotic Reeb orbits for the given almost complex structure
 $\widetilde J$.
To achieve this requirement
the existing studies of such a spectral flow representation in the literature had to use the much bigger 
perturbations of abstract Fredholm operators \emph{in the linearized level}
but \emph{not in the original off-shell nonlinear level}, as well as the perturbations are made
depending on each asymptotic orbit which is not a priori given, especially 
when even its existence is now known. 
(See \cite[Lemma 3.17 \& Appendix C]{wendl:lecture}, for example.) 
 Under such current circumstance,  it is very cumbersome to
develop, for example, a  gluing theory leading to the Kuranishi structures compatible over
different stata of the compactified moduli space of pseudoholomorphic curves entering in the SFT
compactification.

Our Theorem \ref{thm:simple-eigenvalue-intro} cures at least this unsatisfying point on
the spectral flow description of the index in the literature by considering the natural family
$$
\tau \mapsto A^\tau: \Gamma(\gamma_\tau^*\xi) \to \Gamma(\gamma_\tau^*\xi); \quad \gamma_\tau = w(\tau,\cdot)
$$
which connects the asymptotic operator
at $\tau = \pm \infty$ for the given $J$, \emph{after one single perturbation thereof
in the off-shell level} (See \eqref{eq:Atau} for the precise expression.) 
We  believe that the kind of perturbation result stated in Theorem \ref{thm:simple-eigenvalue} and
Theorem \ref{thm:spectral-flow}
will play some role in the construction of Kuranishi structures
on the moduli space of finite energy contact instantons so that certain
natural functor can be defined in our Fukaya-type category of contact manifolds
\cite{kim-oh:contacton-kuranishi}, \cite{kim-oh:rational}. 
(See \cite{bao-honda:kuranishi} for a construction of semi-global Kuranishi structure 
in their definition of contact homology, where the simpleness properties of 
the eigenvalues of the asymptotic operators is utilized in an essential way.)

We also refer readers
to the arXiv version \cite{kim-oh:asymp-analysis} of the present paper for some application 
to the finer study of asymptotic convergence of contact instantons which simplifies the exposition
given in the literature on the pseudoholomorphic curves on symplectization via the systematic
tensorial calculus.

Finally we would like to just mention that the same asymptotic study
can be made by now in a straightforward way by incorporating the boundary condition
as done in \cite{oh-yso:index}, \cite{oh:contacton-transversality},
\cite{oh:contacton-gluing}.

\medskip

The present work has been first announced in the survey paper \cite{oh-kim:survey}
submitted for MATRIX Annal for the IBS-CGP and MATRIX workshop on Symplectic Topology
held for December 5 - 16, 2022.

\medskip
\noindent{\bf Acknowledgement:} We thank MATRIX for providing
an excellent research environment where the present research was initiated.
We also thank Hutchings for attracting our attention to Siefring's
paper \cite{siefring:relative} after we posted
our survey paper \cite{oh-kim:survey} in the arXiv.

\section{Recollection of the analysis of contact instantons}

In this section, we recall the basic Weitzenb\"ock-type identities and the
various analytical results on contact instantons established via covariant
tensorial approach in
\cite{oh-wang:connection}--\cite{oh-wang:CR-map2}, \cite{oh:contacton}--\cite{oh:contacton-Legendrian-bdy},
\cite{oh-savelyev} and  \cite{oh-yso:index}. Our study of asymptotic
operators of contact instantons and of pseudoholomorphic curves
on symplectization also follows this tensorial approach and is based on this
analytical foundation on contact instantons.

Denote by $(\dot\Sigma, j)$ a punctured Riemann surface (including the case of closed Riemann surfaces without punctures).

The following definition is introduced in \cite{oh-wang:CR-map1}.

\begin{defn}[Contact Cauchy-Riemann map]
A smooth map $w:\dot\Sigma\to M$ is called a
\emph{contact Cauchy-Riemann map}
(with respect to the contact triad $(M, \lambda, J)$), if $w$ satisfies the following Cauchy--Riemann equation
$$
\delbar^\pi w:=\delbar^{\pi}_{j,J}w:=\frac{1}{2}(\pi dw+J\pi dw\circ j)=0.
$$
\end{defn}

Recall that for a fixed smooth map $w:\dot\Sigma\to M$,
the triple
$$
(w^*\xi, w^*J, w^*g_\xi)
$$
becomes  a Hermitian vector bundle
over the punctured Riemann surface $(\dot\Sigma,j)$.  This  introduces a Hermitian bundle structure on
$Hom(T\dot\Sigma, w^*\xi)\cong T^*\dot\Sigma\otimes w^*\xi$ over $\dot\Sigma$,
with inner product given by
$$
\langle \alpha\otimes \zeta, \beta\otimes\eta \rangle =h(\alpha,\beta)g_\xi(\zeta, \eta),
$$
where $\alpha, \beta\in\Omega^1(\dot\Sigma)$, $\zeta, \eta\in \Gamma(w^*\xi)$,
and $h$ is the K\"ahler metric on the punctured Riemann surface $(\dot\Sigma, j)$.

Let $\nabla^\pi$ be the contact Hermitian connection.
Combining the pulling-back of this connection and the Levi--Civita connection of the Riemann surface,
we get a Hermitian connection for the bundle $T^*\dot\Sigma\otimes w^*\xi \to \dot\Sigma$, which we will still denote by $\nabla^\pi$ by a slight abuse
of notation. \emph{This is the setting where we apply the harmonic theory and
Weitzenb\"ock formulae to study the global a priori $W^{1,2}$-estimate of
$d^\pi w$:} The smooth map $w$ has an associated $\pi$-harmonic energy density,  the function $e^\pi(w):\dot\Sigma\to \R$, defined by
$$
e^\pi(w)(z):=|d^\pi w|^2(z).
$$
(Here we use $|\cdot|$ to denote the norm from $\langle\cdot, \cdot \rangle$ which should be clear from the context.)

However the contact Cauchy-Riemann equation itself
$\delbar^\pi w = 0$ does not form an elliptic system.
By augmenting the closedness condition $d(w^*\lambda\circ j)=0$ to
contact Cauchy-Riemann map equation $\delbar^\pi w = 0$, we arrive at
an elliptic system \eqref{eq:contact-instanton}
\be\label{eq:contact-instanton}
\delbar^\pi w = 0, \quad d(w^*\lambda \circ j) = 0.
\ee

\subsection{Fundamental equation of contact Cauchy-Riemann maps}
\label{subsec:CRmap}

The following fundamental identity
is derived in \cite{oh-wang:CR-map1}
for whose derivation we refer readers to.

\begin{thm}[Fundamental Equation; Theorem 4.2 \cite{oh-wang:CR-map1}]\label{thm:Laplacian-w}
Let $w$ be a contact Cauchy--Riemann map, i.e., a solution of $\delbar^\pi w=0$.
Then
\be\label{eq:Laplacian-w}
d^{\nabla^\pi} (d^\pi w) = -w^*\lambda\circ j \wedge\left(\frac{1}{2} (\CL_{R_\lambda}J)\,
d^\pi w\right).
\ee
\end{thm}

The following elegant expression of Fundamental Equation in any
\emph{isothermal coordinates} $(x,y)$, i.e., one such that
$z = x+ iy$ provides a complex coordinate of $(\dot \Sigma,j)$
such that $h = dx^2 + dy^2$, will be extremely useful for the study of
higher a priori $C^{k,\alpha}$ H\"older estimates.

\begin{cor}[Fundamental Equation in Isothermal Coordinates]
Let $(x,y)$ be an isothermal coordinates.
Write $\zeta := \pi \frac{\del w}{\del \tau}$
as a section of $w^*\xi \to M$. Then
\be\label{eq:fundamental-isothermal}
\nabla_x^\pi \zeta + J \nabla_y^\pi \zeta
 - \frac{1}{2} \lambda\left(\frac{\del w}{\del x}\right)(\CL_{R_\lambda}J)\zeta
 + \frac{1}{2} \lambda\left(\frac{\del w}{\del y}\right)(\CL_{R_\lambda}J)J\zeta =0.
\ee
\end{cor}

The fundamental equation in cylindrical (or strip-like) coordinates is nothing but the linearization equation of the contact Cauchy-Riemann
equation in the direction
$\frac{\del}{\del\tau}$. This  plays an important role in the derivation
of the exponential decay of the derivatives at cylindrical ends.
(See \cite[Part II]{oh-wang:CR-map2}.)

\subsection{Generic nondegeneracy of Reeb orbits and of Reeb chords}

Nondegeneracy of closed Reeb orbits or of Reeb chords
is fundamental in the Fredholm property of the linearized operator of
contact instanton equations as well as of pseudoholomorphic curves on symplectization.

\subsubsection{The case of closed Reeb orbits}

Let $\gamma$ be a closed Reeb orbit of period $T \neq 0$. In other words,
$\gamma: \R \to M$ is a solution of $\dot x = R_\lambda(x)$ satisfying
$\gamma(T) = \gamma(0)$.
By definition, we can write $\gamma(T) = \phi^T_{R_\lambda}(\gamma(0))$
for the Reeb flow $\phi^T_{R_\lambda}$ of the Reeb vector field $R_\lambda$.
Therefore if $\gamma$ is a closed orbit, then we have
$$
\phi^T_{R_\lambda}(\gamma(0)) = \gamma(0)
$$
i.e.,
$p = \gamma(0)$ is a fixed point of the diffeomorphism $\phi^T_{R_\lambda}$.
Since $\CL_{R_\lambda}\lambda = 0$, $\phi^T_{R_\lambda}$ is a (strict) contact diffeomorphism and so
induces an isomorphism
$$
d\phi^T_{R_\lambda}(p)|_{\xi_p}: \xi_p \to \xi_p
$$
which is the linearization restricted to $\xi_p$ of the Poincar\'e return map.

\begin{defn} We say a $T$-closed Reeb orbit $(T,\lambda)$ is \emph{nondegenerate}
if $d\phi^T_{R_\lambda}(p)|_{\xi_p}:\xi_p \to \xi_p$ with $p = \gamma(0)$ has not eigenvalue 1.
\end{defn}

The following generic nondegeneracy result is proved by
Albers-Bramham-Wendl in \cite{albers-bramham-wendl}. 

\begin{thm}[Albers-Bramham-Wendl] \label{thm:ABW}
Let $(Q,\xi)$ be a contact manifold. Then there exists
a residual subset $\CC^{\text{\rm reg}}(Q,\xi) \subset \CC(Q,\xi)$ such that
for any contact form $\lambda \in \CC^{\text{\rm reg}}(Q,\xi)$ all
Reeb orbits are nondegenerate for $T\neq 0$.
\end{thm}
(The case $T = 0$ can be included as the Morse-Bott nondegenerate case
if we allow the action $T = 0$ by extending the definition of a Reeb trajectory
to \emph{isospeed Reeb chords}  of the pairs $(\gamma,T)$ with
$\gamma:[0,1] \to Q$ with $T = \int \gamma^*\lambda$ as done in
\cite{oh:entanglement1,oh:contacton-transversality}.)

\subsubsection{The case of Reeb chords}
Let $R_0, \, R_1$ be a pair of Legendrian submanifolds.
We first recall the notion of iso-speed Reeb trajectories used in
\cite{oh:entanglement1} and recall the definition of
nondegeneracy of thereof.

Consider contact triads $(Q,\lambda,J)$ and the boundary
value problem
for $(\gamma, T)$ with $\gamma:[0,1] \to Q$
\be\label{eq:chord-equation}
\begin{cases}
\dot \gamma(t) = T R_\lambda(\gamma(t)),\\
\gamma(0) \in R_0, \quad \gamma(1) \in R_1.
\end{cases}
\ee
\begin{defn}[Isospeed Reeb trajectory; Definition 2.1 \cite{oh:contacton-transversality}]
We call a pair $(\gamma,T)$
of a smooth curve $\gamma:[0,1] \to Q$ and $T \in \R$
an \emph{iso-speed Reeb trajectory} if they satisfy
\be\label{eq:isospeed-chords}
\dot \gamma(t) = T R_\lambda(\gamma(t)), \quad \int \gamma^*\lambda = T
\ee
for all $t \in [0,1]$. We call $(\gamma, T)$ an iso-speed closed Reeb orbit
if $\gamma(0) = \gamma(1)$, and an iso-speed Reeb chord of $(R_0,R_1)$
it $\gamma(0) \in R_0$ and $\gamma(1) \in R_1$ from $R_0$ to $R_1$.
\end{defn}

With this definition, we state the corresponding notion of nondegeneracy

\begin{defn}\label{defn:nondegeneracy-chords}
We say a Reeb chord $(\gamma, T)$ of $(R_0,R_1)$ is nondegenerate if
the linearization map $d\phi^T_{R_\lambda}(p): \xi_p \to \xi_p$ satisfies
$$
d\phi^T_{R_\lambda}(p)(T_{\gamma(0)} R_0) \pitchfork T_{\gamma(1)} R_1  \quad \text{\rm in }  \,  \xi_{\gamma(1)}
$$
or equivalently
$$
d\phi^T_{R_\lambda}(p)(T_{\gamma(0)} R_0) \pitchfork T_{\gamma(1)} Z_{R_1} \quad \text{\rm in} \, T_{\gamma(1)}Q.
$$
Here $\phi^t_{R_\lambda}$ is the flow generated by the Reeb vector field $R_\lambda$.
\end{defn}
More generally, we consider the following situation.
We recall the definition of \emph{Reeb trace} $Z_R$ of a Legendrian submanifold $R$, which is defined to be
$$
Z_R: = \bigcup_{t \in \R} \phi_{R_\lambda}^t(R).
$$
(See \cite[Appendix B]{oh:contacton-transversality} for detailed discussion on its genericity.)

\begin{defn}[Nondegeneracy of Legendrian links]
\label{defn:nondegeneracy-links}
Let $\vec{R}=(R_1,\cdots,R_k)$ be a chain of Legendrian submanifolds,
which we call a (ordered) Legendrian link. We say that the Legendrian link $\vec R$ 
is \emph{nondegenerate} if 
it satisfies
$$
Z_{R_i} \pitchfork R_j
$$
for all $i, \, j= 1,\ldots, k$.
\end{defn}

We now provide the off-shell framework for the proof of nondegeneracy
in general. Denote $\CL(Q)=C^\infty(S^1,Q)$ the space of loops $z: S^1 = \R /\Z \to Q$.
We denote by
$$
\CC(Q,\xi)
$$
the set of contact forms of $(Q,\xi)$ equipped with $C^\infty$-topology.
We denote by $\LL(Q;R_0,R_1)$ the space of paths
$$
\gamma: ([0,1], \{0,1\}) \to (Q;R_0,R_1).
$$
We consider the assignment
\be\label{eq:Phi-TR}
\Phi: (T,\gamma,\lambda) \mapsto \dot \gamma - T \,R_\lambda(\gamma)
\ee
as a section of the Banach vector bundle over
$$
(0,\infty) \times \CL^{1,2}(Q;R_0,R_1) \times \CC(Q,\xi)
$$
where $\CL^{1,2}(Q;R_0,R_1)$
is the $W^{1,2}$-completion of $\CL(Q;R_0,R_1)$. We have
$$
\dot \gamma - T\, R_\lambda(\gamma) \in \Gamma(\gamma^*TQ; T_{\gamma(0)}R_0, T_{\gamma(1)}R_1).
$$
We  define the vector bundle
$$
\CL^2(R_0,R_1) \to (0,\infty) \times \CL^{1,2}(Q;R_0,R_1) \times \CC(Q,\xi)$$
whose fiber at $(T,\gamma,\lambda)$ is $L^2(\gamma^*TQ)$. We denote by
$\pi_i$, $i=1,\, 2, \, 3$ the corresponding projections as before.

We denote $\mathfrak{Reeb}(M,\lambda;R_0,R_1) = \Phi_\lambda^{-1}(0)$,
where
$$
\Phi_\lambda: = \Phi|_{ (0,\infty) \times \CL^{1,2}(Q;R_0,R_1) \times \{\lambda\}}.
$$
Then we have
$$
\mathfrak{Reeb}(\lambda;R_0,R_1) =  \Phi_\lambda^{-1}(0) = \mathfrak{Reeb}(Q,\xi) \cap \pi_3^{-1}(\lambda).
$$
The following relative version of Theorem \ref{thm:ABW} is proved in
\cite[Appendix B]{oh:contacton-transversality}.

\begin{thm} [Perturbation of contact forms;
Theorem B.3 \cite{oh:contacton-transversality}]
\label{thm:Reeb-chord-lambda}
Let $(Q,\xi)$ be a contact manifold. Let  $(R_0,R_1)$ be a pair of Legendrian submanifolds
allowing the case $R_0 = R_1$.  There
exists a residual subset $\CC^{\text{\rm reg}}_1(Q,\xi) \subset \CC(Q,\xi)$
such that for any $\lambda \in \CC^{\text{\rm reg}}_1(Q,\xi)$ all
Reeb chords from $R_0$ to $R_1$ are nondegenerate for $T \neq 0$ and
Bott-Morse nondegenerate when $T = 0$.
\end{thm}
The following theorem is also proved in \cite{oh:contacton-transversality}.

\begin{thm}[Perturbation of boundaries;
Theorem B.10 \cite{oh:contacton-transversality}]
\label{thm:Reeb-chords-MB}
Let $(Q,\xi)$ be a contact manifold. Let  $(R_0,R_1)$ be a pair of Legendrian submanifolds allowing the case $R_0 = R_1$.
 For a given contact form $\lambda$ and $R_1$,
there exists a residual subset
$$
R_0 \in{\mathcal Leg}^{\text{\rm reg}}(Q,\xi) \subset {\mathcal Leg}(Q,\xi)
$$
of Legendrian submanifolds such that for all $R_0 \in {\mathcal Leg}^{\text{\rm reg}}(Q,\xi)$ all
Reeb chords from $R_0$ to $R_1$ are nondegenerate for $T \neq  0$ and Morse-Bott nondegenerate when $T = 0$.
\end{thm}

We refer readers to \cite[Appendix B]{oh:contacton-transversality} for the proofs of these results.

\subsection{Exponential asymptotic $C^\infty$ convergence of contact instantons}
\label{subsec:exponential-convergence}

In this section, we recall the asymptotic behavior of contact instantons
on the Riemann surface $(\dot\Sigma, j)$ associated with a metric $h$ with \emph{cylinder-like ends} for the closed string context and with
\emph{strip-like ends} for the open string context.

We assume there exists a compact set $K_\Sigma\subset \dot\Sigma$,
such that $\dot\Sigma-\Int(K_\Sigma)$ is a disjoint union of interior-punctured disks
 each of which is isometric to the half cylinder $[0, \infty)\times S^1$ or
 the half strip $(-\infty, 0]\times S^1$, where
the choice of positive or negative strips depends
on the choice of analytic coordinates at the punctures.
We denote by $\{p^+_i\}_{i=1, \cdots, l^+}$ the positive punctures, and by $\{p^-_j\}_{j=1, \cdots, l^-}$ the negative punctures. Here $l=l^++l^-$. The case of boundary-punctured disks is similar.
Then we denote by $\phi^{\pm}_i$ such cylinder-like or strip-like coordinates depending on
whether they are boundary or interior punctures. 

We separately describe the cases of interior punctures and of boundary punctures.
We first mainly state our assumptions for the study of the behavior of boundary punctures.
(The case of interior punctures is treated in \cite[Section 6]{oh-wang:CR-map1} and will be briefly mentioned at the end of this section.)

\begin{defn} Let $\dot\Sigma$ be a punctured Riemann surface of genus zero with boundary punctures
$\{p^+_i\}_{i=1, \cdots, l^+}\cup \{p^-_j\}_{j=1, \cdots, l^-}$ equipped
with a metric $h$ with \emph{strip-like ends} outside a compact subset $K_\Sigma$.
Let
$w: \dot \Sigma \to M$ be any smooth map with Legendrian boundary condition.
We define the total $\pi$-harmonic energy $E^\pi(w)$
by
\be\label{eq:pienergy}
E^\pi(w) = E^\pi_{(\lambda,J;\dot\Sigma,h)}(w) = \frac{1}{2} \int_{\dot \Sigma} |d^\pi w|^2
\ee
where the norm is taken in terms of the given metric $h$ on $\dot \Sigma$ and the triad metric on $M$.
\end{defn}

Throughout this section, we work locally near one boundary puncture
$p$, i.e., on a punctured semi-disc
$D^\delta(p) \setminus \{p\}$. By taking the associated conformal coordinates $\phi^+ = (\tau,t)
:D^\delta(p) \setminus \{p\} \to [0, \infty)\times [0,1]$ such that $h = d\tau^2 + dt^2$,
we need only look at a map $w$ defined on the semi-strip
 $[0, \infty)\times [0,1]$ without loss of generality.

Under the nondegeneracy hypothesis from Definition
\ref{defn:nondegeneracy-links} and the transversality hypothesis,
the exponential $C^\infty$ convergence from is derived 
from the subsequence convergence and charge-vanishing result in \cite{oh-wang:CR-map1}, \cite{oh-yso:index}.
For readers' convenience and for the self-containedness of the paper,
we recall the subsequence and charge in Appendix \ref{sec:charge-vanishing},
and the exponential convergence result in this subsection.

Suppose that the tuple $\vec R= (R_0, \ldots, R_k)$ are
transversal in the sense all pairwise Reeb chords are nondegenerate. In particular we assume
that the tuples are pairwise disjoint. 

Firstly, we state the following $C^0$-exponential convergence from \cite{oh-wang:CR-map2}
for the closed string and \cite{oh-yso:index} for the open string case.

\begin{prop}[Proposition 11.23 \cite{oh-wang:CR-map2}, Proposition 6.5 \cite{oh-yso:index}]
There exist some constants $C > 0$, $\delta>0$ and $\tau_0$ large such that for any $\tau>\tau_0$,
\beastar
\|d\left( w(\tau, \cdot), \gamma(\cdot) \right) \|_{C^0([0,1])} &\leq& C \, e^{-\delta \tau}
\eeastar
where we have $0 < \delta < |\mu_1|$ with the first negative eigenvalue $\mu_1$ of the
asymptotic operator $A_{(\lambda,J,\nabla)}^{\pi}$ of $\gamma$.
\end{prop}

Once the above $C^0$-exponential decay is established the proof of
$C^\infty$-exponential convergence $w(\tau,\cdot) \to \gamma$
by establishing the $C^\infty$-exponential decay of
$$
dw - R_\lambda(w)\, dt.
$$
The proof of the latter decay in \cite{oh-yso:index} is carried out
by an \emph{alternating boot strap argument} by decomposing
$$
dw = d^\pi w + w^*\lambda\, R_\lambda
$$
as follows. Let $z = x+i y$ be any isothermal coordinates on $(D_2,\del D_2) \subset (\dot \Sigma,\del \dot \Sigma)$
adapted to the boundary, i.e., satisfying that $\frac{\del}{\del x}$ is tangent to $\del \dot \Sigma$.
We set
\beastar
\zeta & : = & d^\pi w(\del_x), \\
 \chi &: = & \lambda\left(\frac{\del w}{\del y}\right)
+ \sqrt{-1} \lambda\left(\frac{\del w}{\del x}\right)
\eeastar
Then we show that the fundamental equation \eqref{eq:Laplacian-w} is transformed into the following system of equations for the pair
$(\zeta,\alpha)$
\be\label{eq:equation-for-zeta0}
\begin{cases}\nabla_x^\pi \zeta + J \nabla_y^\pi \zeta
+ \frac{1}{2} \lambda(\frac{\del w}{\del y})(\CL_{R_\lambda}J)\zeta - \frac{1}{2} \lambda(\frac{\del w}{\del x})(\CL_{R_\lambda}J)J\zeta =0\\
\zeta(z) \in TR_i \quad \text{for } \, z \in \del D_2
\end{cases}
\ee
and
\be\label{eq:equation-for-alpha}
\begin{cases}
\delbar \chi = \frac{1}{2}|\zeta|^2 \\
\chi(z) \in \R \quad \text{for } \, z \in \del D_2
\end{cases}
\ee
for some $i = 0, \ldots, k$. With this coupled system of equations for $(\zeta,\chi)$
at our disposal, the proof follows the alternating boot-strap 
between $\zeta$ and $\chi$ similarly as in the proof of
of higher regularity results carried out
by the alternating boot strap argument  in \cite{oh:contacton-Legendrian-bdy,oh-yso:index}.

Combining this and elliptic $C^{k,\alpha}$-estimates
given in \cite{oh-wang:CR-map1,oh-yso:index}, the proof of $C^\infty$-convergence of
$w(\tau,\cdot) \to \gamma$ as $\tau \to \infty$ is completed.

So far we have recollected various foundational analytic results
on contact instantons both in the closed and in the open string
context, which will be used later in our study of asymptotic operators.
 Since the same arguments can be applied to the open string
case with by now straightforward incorporation of the boundary
condition exercised in \cite{oh-yso:index}, \cite{oh:contacton-gluing},
\cite{oh:contacton-transversality} and \cite{oh-yso:spectral},
\emph{we will focus on the case of closed strings to streamline
our study of asymptotic operators and to highlight main points
our approach in the rest of the paper.}

\subsection{Exponential convergence of pseudoholomorphic
curves on symplectization}

Finally we make a brief mention on how the exponential
convergence for pseudoholomorphic curves on symplectization
follows from that of contact instantons now.
We consider the symplectization
$$
M = Q \times \R, \quad \omega = d(e^s \pi^*\lambda) = e^s (ds \wedge \pi^*\lambda + d\pi^*\lambda)
$$
of the contact manifold $(Q,\xi)$ equipped with contact form $\lambda$.

On $Q$, the Reeb vector field $R_\lambda$ associated to the contact
form $\lambda$ is the unique vector field $X = :R_\lambda$ satisfying
\be\label{eq:Liouville} X \rfloor \lambda = 1, \quad X \rfloor
d\lambda = 0.
\ee
We call $(y,s)$ the cylindrical coordinates.
On the cylinder $[0, \infty) \times Q \subset (-\infty,\infty) \times Q$,
we have the natural splitting \eqref{eq:TM-splitting} of the tangent bundle $TM$.

Now we describe a special family of almost complex structure compatible to
the given cylindrical structure of $M$.

\begin{defn}\label{defn:lambdcompatibleJ} An almost complex structure $J$ on $ Q \times (0,\infty)$ is called
\emph{$\lambda$-compatible} if it is split into
$$
J = J_\xi \oplus J_0: TM \cong \xi \oplus \R^2  \to TM \cong \xi \oplus \R^2
$$
where $J|_\xi$ is compatible to $d\lambda|_{\xi}$ and $j: \R^2 \to \R^2$
maps $\frac{\del}{\del s}$ to $R_\lambda$.
\end{defn}

For our purpose, we will need to consider a family of symplectic forms
to which the given $J$ is compatible and their associated metrics.
For any $\lambda$-compatible $J$, the $J$-compatible metric associated
to $\omega$ is expressed as
\be\label{eq:gJ}
g_{(\omega,J)}  = ds^2 + g_Q
\ee
on $\R \times Q$.

Now we regard the triple $(\omega,J,g_{(\omega,J)})$ be an almost
Hermitian manifold near the level surface $s = 1$. We then fix the canonical
connection $\nabla$ associated to $(\omega,J,g_{(\omega,J)})$.
(See Appendix \ref{sec:connection}.)

The following
is a general property of the canonical connection.

\begin{prop}\label{prop:TJYY} Let $(W,\omega, J)$ be an almost Hermitian manifold
and $\nabla$ be the canonical connection. Denote by $T$ be its torsion tensor. Then
\be\label{eq:TJYY}
T(JY,Y) = 0
\ee
for all vector fields $Y$ on $W$.
\end{prop}
Consider the decomposition \eqref{eq:TM-splitting}
and the canonical connection $\widetilde \nabla$ on $Q \times \R$, which in particular
is $J$-linear.
Recalling the expression $u = (w,f)$ with
$f = s\circ u$ and $w = \pi_Q \circ u$ for each map $u: \dot \Sigma \to Q \times \R$, 
we know that if it is $J$-holomorphic, it satisfies
$$
\delbar^\pi w = 0, \quad w^*\circ j = df
$$
on $\dot \Sigma$. In particular $w$ is a contact instanton.
Then we have already shown the exponential convergence $w$ to
the Reeb orbit $w(\cdot, t)$.

For the convergence of $f$, we use the equation
$$
w^*\lambda \circ j = df
$$
which is equivalent to
$$
w^*\lambda = - df\circ j
$$
By taking differential, we obtain
$$
-d(df\circ j) = w^*d\lambda = \frac12 |d^\pi w|^2 \, d\tau \wedge dt.
$$
Firstly, this equation shows that $f$ is a subharmonic function. More explicitly, we derive
satisfies
\be\label{eq:a-asymptotics}
\frac{\del^2 f}{\del \tau^2} + \frac{\del^2 f}{\del t^2} - \frac12 |d^\pi w|^2 = 0
\ee
where we know  from Theorem \ref{thm:subsequence} that the convergence $\frac12 |d^\pi w|^2 \to T^2$ is exponentially fast.
This immediately
gives rise to the following exponential convergence of the radial component
which will complete the study of asymptotic convergence property of
finite energy pseudoholomorphic planes in symplectization.

\begin{prop}[Exponential convergence of radial component] Let
$u=(w,f)$ be a finite energy $\widetilde J$-holomorphic plane in
$Q\times \R$. Then we have convergence
$$
df \to T\, d\tau
$$
exponentially fast. More precisely, there exists some $c \in \R$ such that
$$
|f(\tau,t) - (T\tau + c)| \to 0
$$
as $\tau \to \infty$ exponentially fast.
\end{prop}
\begin{proof} We have already established $w^*\lambda \to T\, dt$ before
in Corollary \ref{cor:Q=0}.
By composing by $j$, the statement follows.
\end{proof}

\section{Covariant linearization operator and its Fredholm theory}
\label{sec:Fredholm}

In this section, we work out the Fredholm theories of
pseudoholomorphic curves on symplectization by recalling the
exposition given in \cite{oh:contacton} for the case of contact instantons and
the one given in \cite{oh-savelyev} for the case of $\mathfrak{lcs}$-instantons.
The zero-temperature limit of the latter also provides the relevant Fredholm theory
for pseudoholomorphic curves just by incorporating the presence of the $\R$-factor
in the product $M = Q^{2n-1} \times \R$ into that of contact instantons.
Explanation of this point is now in order.

Let $\Sigma$ be a closed Riemann surface and $\dot \Sigma$ be its
associated punctured Riemann surface.  We allow the set of whose punctures
to be empty, i.e., $\dot \Sigma = \Sigma$.

We recall the splitting $TM = \xi \oplus \CV$ from \eqref{eq:TM-splitting}.
Using this splitting, we would like to regard the assignment $u \mapsto \delbar_J u$ which can be
decomposed into
$$
u = (w,f) \mapsto \left(\delbar^\pi w, w^*\lambda \circ j - f^*ds\right) =: \Upsilon(u)
$$
for a map $w: \dot \Sigma \to Q$ as a section of the (infinite dimensional) vector bundle
over the space of maps of $w$. In this section, we lay out the precise relevant off-shell framework
of functional analysis. We recall the definition of the $(0,1)$-projection of
the covariant differential $d^{\nabla^\pi}$
$$
\delbar^{\nabla^\pi} := \frac12\left(\nabla^\pi + J \nabla^\pi \circ j\right).
$$
(We recall readers the definition of general covariant differential in Appendix \ref{sec:covariant-differential},
which in particular applies to $d^{\nabla^\pi}$.)

We decompose
$$
\Upsilon(u) = (\Upsilon_1(u), \Upsilon_2(u))
$$
into the $\xi$ and the Reeb components respectively. Then we have the formulae
\be\label{eq:Upsilon12}
\Upsilon_1(u) = \delbar^\pi u, \quad \Upsilon_2(u) = w^*\lambda \circ j - f^*ds.
\ee

\begin{thm}[Theorem 10.1 \cite{oh-savelyev}] \label{thm:linearization} We decompose $d\pi = d^\pi w + w^*\lambda\otimes R_\lambda$
and $Y = Y^\pi + \lambda(Y) R_\lambda$, and $X = (Y, v) \in \Omega^0(w^*T(Q \times \R))$.
Denote $\kappa = \lambda(Y)$ and $b = ds(v)$. Then we have
\bea
D\Upsilon_1(u)(Y,v) & = & \delbar^{\nabla^\pi}Y^\pi + B^{(0,1)}(Y^\pi) +  T^{\pi,(0,1)}_{dw}(Y^\pi) \nonumber\\
&{}& \quad + \frac{1}{2}\kappa \cdot  \left(\CL_{R_\lambda}J)J(\del^\pi w\right)
\label{eq:Dwdelbarpi}\\
D\Upsilon_2(u)(Y,v) & = &  w^*(\CL_Y \lambda) \circ j- \CL_v ds = d\kappa \circ j - db
+ w^*(Y \rfloor d\lambda) \circ j
\nonumber\\
\label{eq:Dwddot}
\eea
where $B^{(0,1)}$ and $T_{dw}^{\pi,(0,1)}$ are the $(0,1)$-components of $B$ and
$T_{dw}^\pi$, where $B, \, T_{dw}^\pi: \Omega^0(w^*TQ) \to \Omega^1(w^*\xi)$ are
zero-order differential operators given by
$$
B(Y) =
- \frac{1}{2}  w^*\lambda \left((\CL_{R_\lambda}J)J Y\right)
$$
and
$$
T_{dw}^\pi(Y) = \pi T(Y,dw)
$$
respectively.
\end{thm}

Now we consider a punctured Riemann surface $\dot \Sigma$.
We consider some choice of weighted Sobolev spaces
$$
\CW^{k,p}_{\delta;\eta}\left(\dot \Sigma,Q \times \R;\vec \gamma^+, \vec \gamma^-\right)
$$
as the off-shell function space and linearize the map
$$
(w,\widetilde f) \mapsto \left(\delbar^\pi w,  d\widetilde f\right).
$$
This linearization operator then becomes cylindrical in cylindrical
coordinates near the punctures.
The Fredholm property of the linearization map
$$
D\Upsilon_{(\lambda,J)}(u): \Omega^0_{k,p;\delta}(u^*T(Q \times \R);J;\gamma^+,\gamma^-) \to
\Omega^{(0,1)}_{k-1,p;\delta}(w^*\xi) \oplus \Omega^{(0,1)}_{k-1,p}(u^*\CV)
$$
and its index is computed in \cite{oh:contacton}, \cite{oh:contacton-transversality}
and in \cite{oh-yso:index} respectively. 

We briefly recall the aforementioned Fredholm property here.
We have the decomposition
\be\label{eq:Omega0-decompose}
\Omega^0_{k,p;\delta}(w^*T(Q \times \R);J;\gamma^+,\gamma^-) =
\Omega^0_{k,p;\delta}(w^*\xi) \oplus \Omega^0_{k,p;\delta}(u^*\CV)
\ee
and again the operator
\be\label{eq:DUpsilonu}
D\Upsilon_{(\lambda,J)}(u): \Omega^0_{k,p;\delta}(w^*T(Q \times \R);J;\gamma^+,\gamma^-) \to
\Omega^{(0,1)}_{k-1,p;\delta}(w^*\xi) \oplus \Omega^{(0,1)}_{k-1,p;\delta}(u^*\CV)
\ee
which is decomposed into
$$
D\Upsilon_1(u)(Y,v)\oplus D\Upsilon_2(u)(Y,v)
$$
where the summands are given as in
\eqref{eq:Dwdelbarpi} and \eqref{eq:Dwddot} respectively. 

In terms of the decomposition \eqref{eq:Omega0-decompose},
the linearized operator $D\Upsilon_{(\lambda, J)}(u)$ can be written as
in the following succinct matrix form
\be\label{eq:matrix-form}
\left(
\begin{matrix}
 \delbar^{\nabla^\pi} + B^{(0,1)} +  T^{\pi,(0,1)}_{dw} &, & \frac{1}{2}(\cdot) \cdot
  \left((\CL_{R_\lambda}J)J(\del^\pi w)\right)\\
\left((\cdot)^\pi \rfloor d\lambda\right)\circ j &, & \delbar
\end{matrix}
\right).
\ee
For the calculation of the index of the linearized operator, we would like to
homotope to the block-diagonal form, i.e., into the direct sum operator
$$
\left(\delbar^{\nabla^\pi} + T^{\pi,(0,1)}_{dw}  + B^{(0,1)}\right)\oplus \delbar
= D_w \delbar^\pi \oplus \delbar
$$
via a \emph{continuous} path of Fredholm operators. The
Fredholm property of all elements in the path used in \cite{oh:contacton},
\cite{oh-savelyev} relies on the
following asymptotic property of the off-diagonal terms.
This was implicitly used in the calculation of the index in
\cite{oh:contacton}, \cite{oh-savelyev}, and its proof is given in \cite[Proposition 13.6]{oh-kim:survey}.
For readers' convenience, we reproduce its proof here to show how the asymptotic exponential
convergence and the explicit formula of the linearized operator
enter in the asymptotic property.

\begin{prop}\label{prop:off-diagonal}
The off-diagonal terms decay exponentially fast as $|\tau|\to \infty$. In particular
the off-diagonal term is a compact operator relatively to the diagonal term.
\end{prop}
\begin{proof} For the $(1,2)$-term of the matrix \eqref{eq:matrix-form}, we derive
$$
\left(\frac{\del w}{\del \tau}\right)^\pi + J \left(\frac{\del w}{\del t}\right)^\pi = 0
$$
from the equation $\delbar^\pi w = 0$ by evaluating it against $\frac{\del}{\del \tau}$.
 
 Therefore we have
$$
\del^\pi w\left(\frac{\del}{\del \tau}\right) = \frac12\left(\left(\frac{\del w}{\del \tau}\right)^\pi - J \left(\frac{\del w}{\del t}\right)^\pi \right)
= - J \left(\frac{\del w}{\del t}\right)^\pi.
$$
By the exponential convergence $\frac{\del w}{\del t} \to T R_\lambda(\gamma_\infty(t))$, we derive
$$
J \del^\pi w\left(\frac{\del}{\del \tau}\right)
= \left(\frac{\del w}{\del t}\right)^\pi \to 0
$$
since $\frac{\del w}{\del t} \to T R_\lambda$. 

For the $(2,1)$-term, we evaluate
\beastar
(Y^\pi \rfloor d\lambda)\circ j\left(\frac{\del}{\del \tau}\right) & = & d\lambda\left(Y, \frac{\del w}{\del t}\right)\\
(Y^\pi \rfloor d\lambda)\circ j\left(\frac{\del}{\del t}\right) & = & - d\lambda\left(Y, \frac{\del w}{\del \tau}\right).
\eeastar
Therefore we have derived
$$
(Y^\pi \rfloor d\lambda)\circ j = d\lambda\left(Y, \frac{\del w}{\del t}\right) \, d\tau
- d\lambda\left(Y, \frac{\del w}{\del \tau}\right)\, dt.
$$
This proves that
$$
\left((\cdot)^\pi \rfloor d\lambda\right)\circ j \to  d\lambda\left(\cdot, \frac{\del w}{\del t}\right) \, d\tau - d\lambda\left(\cdot , \frac{\del w}{\del \tau}\right)\, dt
$$
as $|\tau| \to \infty$. Since
$\frac{\del w}{\del t} \to T R_\lambda$, the first term converges to zero, 
and the second term converges to
$$
- d\lambda (\cdot, J TR_\lambda) = T d\lambda \left(\cdot, \frac{\del}{\del s}\right) = 0,
$$
all exponentially fast. 

Combining the two, we have proved that 
the off-diagonal term converges to the zero operator exponentially fast. The last statement
about the relatively compactness is an immediate consequence of this exponential decay.

This finishes the proof.
\end{proof}

Because of this asymptotic vanishing, the path
\be\label{eq:continuous-path}
s \in [0,1] \mapsto
\left(
\begin{matrix}
 \delbar^{\nabla^\pi} + B^{(0,1)} +  T^{\pi,(0,1)}_{dw} &, & \frac{1-s}{2}(\cdot) \cdot
  \left((\CL_{R_\lambda}J)J(\del^\pi w)\right)\\
(1-s)\left((\cdot)^\pi \rfloor d\lambda\right)\circ j &, & \delbar
\end{matrix}
\right) =: L_s
\ee
carries the same asymptotic operator and hence
is a continuous path of Fredholm operators
$$
L_s:  \Omega^0_{k,p;\delta}(w^*T(Q \times \R);J;\gamma^+,\gamma^-) \to
\Omega^{(0,1)}_{k-1,p;\delta}(w^*\xi) \oplus \Omega^{(0,1)}_{k-1,p;\delta}(u^*\CV)
$$
such that $L_0 = D\Upsilon_{(\lambda,J)}(u)$ and
$$
L_1 =   \left(\delbar^{\nabla^\pi} + B^{(0,1)} +  T^{\pi,(0,1)}_{dw}\right) \oplus \delbar = D_w \delbar^\pi \oplus \delbar. 
$$
Therefore we have only to
compute the index of the diagonal operator $L_1$ which was
given in \cite{oh:contacton}, \cite{oh-savelyev}.

\section{Definition of asymptotic operator and its covariant formulae}
\label{sec:asymptotic-operators}

Recalling that for any $\widetilde J$-holomorphic curve $(w,f)$ on the symplectization,
$w$ is a contact instanton for $J$ on $Q$. Furthermore we have
$$
(w,f)^*T(Q \times \R) = w^*TQ \oplus f^*T\R
= w^*\xi \oplus \span_\R \left\{\frac{\del}{\del s}, R_\lambda \right\}.
$$
The last splitting is respected
by the \emph{canonical connection} of the almost Hermitian manifold
$$
(Q\times \R, d\lambda + ds \wedge \lambda, \widetilde J).
$$
Indeed, we have
$$
\nabla^{\text{\rm can}} = \nabla^\pi \oplus \nabla_0
$$
where $\nabla^\pi = \pi \nabla|_{\xi}$ and $\nabla_0$ is the trivial
connection on $\span_\R\{\frac{\del}{\del s}, R_\lambda\}$.

\begin{rem} Here again we would like to emphasize the usage of
the canonical connection of the almost Hermitian manifold, not
the Levi-Civita connection, admits this splitting.
\end{rem}

Now we study a finer analysis of the asymptotic behavior along the Reeb orbit. Our discussion thereof is close to the one given in
\cite[Section 11.2 \& 11.5]{oh-wang:CR-map2} where the more general
Morse-Bott case is studied.

For this purpose, we  evaluate the linearization operator $D\Upsilon$ against $\frac{\del}{\del \tau}$. We have already checked the
off-diagonal terms of the matrix representation of $D\Upsilon(w)$
decays exponentially fast in the direction $\tau$
in the previous section and so we have only to examine the diagonal terms $D\Upsilon_1(w)$ and $D\Upsilon_2(w)$.

First we consider $D\Upsilon_2$ and rewrite
$$
D\Upsilon_2(u) =\delbar = \frac12(\del_\tau + i \del_t).
$$
Therefore we have the asymptotic operator
\be\label{eq:Aperp}
A^\perp{(\lambda,J,\nabla)} : = - i \del_t
\ee
which does not depend on the choice of $J \in \CJ_\lambda(Q,\xi)$.
The eigenfunction expansions for this operator is nothing but
the standard Fourier series for $f \in L^2(S^1,Q)$.

This being said, we now focus on the $Q$-component $D\Upsilon_1(u)$ of the
asymptotic operator, and compute
\be\label{eq:Dupsilon1-ddtau}
D\delbar^\pi (w)\left(\frac{\del}{\del \tau}\right) = \frac12(\nabla_\tau^\pi +  J \nabla_t^\pi)
+ T_{dw}^{\pi,(0,1)}\left(\frac{\del}{\del \tau}\right)
+ B^{(0,1)}\left(\frac{\del}{\del \tau}\right).
\ee
In fact, this is nothing but the left hand side of \eqref{eq:fundamental-isothermal}
by the calculation of the torsion term, which we omit since we have already have
the formula \eqref{eq:fundamental-isothermal}.
We write
$$
D\delbar^\pi (w)\left(\frac{\del}{\del \tau}\right) =  \frac12\left(\nabla_\tau^\pi - A^\tau_{(\lambda,J,\nabla)}\right).
$$
and define the family of operators
$$
A^\tau = A^{\tau}_{(\lambda,J,\nabla)}: \Gamma(w_\tau^*\xi) \to \Gamma(w_\tau^*\xi)
$$
given by the formula
\be\label{eq:Atau}
A^\tau: = - J \nabla_t^\pi - \left(
 T_{dw}^{\pi,(0,1)}\left(\frac{\del}{\del \tau}\right)
+ B^{(0,1)}\left(\frac{\del}{\del \tau}\right)\right).
\ee
Thanks to the exponential convergence of $w_\tau= w(\tau, \cdot) \to \gamma_\pm$ as $\tau \to \pm \infty$,
we can take the limit of the conjugate operators
\be\label{eq:Atau-conjugate}
\Pi_\tau^\infty A^\tau_{(\lambda,J,\nabla)}(\Pi_\tau^\infty)^{-1}: \Gamma(\gamma_\pm^*\xi) \to \Gamma(\gamma_\pm^*\xi)
\ee
as $\tau \to \pm \infty$ respectively, where $\Pi_\tau^\infty$ is the parallel transport
along the short geodesics from $w(\tau,t)$ to $w(\infty,t)$. This conjugate is defined
for all sufficiently large $|\tau|$.

Since the discussion at $\tau = -\infty$ will be the same, we will focus our discussion
on the case at $\tau = + \infty$ from now on.

\begin{defn}[Asymptotic operator]\label{defn:asymptotic-operator}
 Let $(\tau,t)$ be
the cylindrical (or strip-like) coordinate, and let $\nabla^\pi$ be
the almost Hermitian connection on $w^*\xi$ induced by
the contact triad connection $\nabla$ of $(Q,\lambda,J)$.
We define the \emph{asymptotic operator} of a contact
instanton $w$ to be the limit operator
\be\label{eq:asymptotic-operator}
A^\pi_{(\lambda,J,\nabla)}: = \lim_{\tau \to +\infty}\Pi_\tau^\infty A^\tau_{(\lambda,J,\nabla)}(\Pi_\tau^\infty)^{-1}.
\ee
\end{defn}
Obviously we can define the asymptotic operator at negative punctures in the similar way.

\subsection{Asymptotic operator in contact triad connection}

Now we find the formula for the above
limit operator \emph{with respect to the contact triad connection}.
Since $T(R_\lambda, \cdot) = 0$, $\left(\frac{\del w}{\del \tau}\right)^\pi
= - J \left(\frac{\del w}{\del t}\right)^\pi$ and
$\frac{\del w}{\del t}(\tau, \cdot) \mapsto T R_\lambda$ exponentially fact, we obtain
$$
T_{dw}^{\pi,(0,1)}(\frac{\del}{\del \tau}) = T^\pi\left(\frac{\del w}{\del \tau},\cdot\right) \to  0.
$$
On the other hand, we have
$$
2 B^{(0,1)}\left(\frac{\del}{\del \tau}\right) = -\frac12 \lambda\left(\frac{\del w}{\del \tau}\right) \CL_{R_\lambda}J
-\frac12 \lambda\left(\frac{\del w}{\del t}\right) J \CL_{R_\lambda}J \\
$$
This converges to $\frac{T}{2} J \nabla_{R_\lambda}$ since $w^*\lambda \to T \, dt$
as $\tau \to \infty$.

This immediately gives rise to the following simple explicit formula
 for the asymptotic operator.

\begin{prop}\label{prop:A} Let $\nabla$ be the contact triad connection
associated to any compatible pair $(\lambda, J)$. Then the asymptotic operator $A^\pi_{(\lambda,J,\nabla)}$ is given by
$$
A^\pi_{(\lambda,J,\nabla)} = - J \nabla_t + \frac{T}2 \CL_{R_\lambda}J J
$$
In particular, it induces a (real) self-adjoint operator on $(\xi, g|_\xi)$
 with respect to the triad metric $g$.
 \end{prop}
\begin{proof} We have already shown the formula above. The symmetry of the opeartor
$-J \nabla_t$ can be directly checked from the $J$-linearity of $\nabla^\pi$ by
the integration by parts.

Now we recall from  \cite{blair:book}
that  that the operator $\CL_{R_\lambda}J J$ defines a symmetric operator
on $\Gamma(\gamma^*\xi)$ by the following general identity and hence so is $A^\pi_{(\lambda,J,\nabla)}$.

\begin{lem}[Lemma 6.2 \cite{blair:book}]
The linear maps $\CL_{R_\lambda}J J$ and $\CL_{R_\lambda}J$ are
(pointwise) symmetric with respect to the triad metric.
\end{lem}

This finishes the proof of the proposition.
\end{proof}

Proposition \ref{prop:A} enables us to derive the following.

\begin{prop}\label{prop:nablaRlambdaY} We have
\be\label{eq:nablaRlambdaY}
\nabla_{R_\lambda}Y = [R_\lambda,Y] +  \frac12 (\CL_{R_\lambda}J) JY
= \CL_{R_\lambda} Y + \frac12 (\CL_{R_\lambda}J) JY
\ee
for all $Y \in \xi$.
\end{prop}
\begin{proof} We recall the formula
$$
\nabla_Y R_\lambda = \frac12 (\CL_{R_\lambda} J)JY
$$
from \eqref{eq:nablaYRlambda}. By the definition of the torison,
we also have
$$
\nabla_Y R_\lambda = \nabla_{R_\lambda} Y + [Y,R_\lambda] + T(Y,R_\lambda).
$$
By combining these with the torsion property $T(\cdot,R_\lambda) = 0$ of the triad connection,
we obtain the first equality. The second is just the definition of the Lie derivative $\CL_{R_\lambda}$
acted upon vector fields.
\end{proof}

We note that the right hand side formula in \eqref{eq:nablaRlambdaY}
is canonically defined depending only on $\lambda$ and $J$
\emph{independent of the choice of connection.}
In other words, we can derive the first variation of
the operator $\nabla_{R_\lambda}|_\xi$ associated to the triad
connection of $(Q,\lambda,J)$ with respect to the compatible pair
$(\lambda,J)$ or with respect to $J$ when $\lambda$ is fixed.

Now in the context of pseudoholomorphic curves on symplectization,
we proceed the procedure by decomposing the full asymptotic operator
$A^\pi_{(\lambda,J,\nabla)}$ of the
pseudoholomorphic curves to be the operator
$$
A^\pi_{(\lambda,J,\nabla)}: \gamma^*\xi \oplus \C \to  \gamma^*\xi \oplus \C
$$
into
$$
A_{(\lambda,J,\nabla)} = A^\pi_{(\lambda,J,\nabla)} \oplus
A^\perp_{(\lambda,J,\nabla)}.
$$
Here $\C$ stands for the pull-back bundle
$$
(\pi \circ u_\tau)^*\left(\R \left\{\frac{\del}{\del s}, R_\lambda \right\}\right)
 = w_\tau^*\left(\R \left\{\frac{\del}{\del s}, R_\lambda \right\}\right)
\cong \R \left \{\frac{\del}{\del s}, R_\lambda \right\}
$$
which is canonically trivialized, and hence may be regarded as a vector bundle over
the curves $w_\tau$ on $Q$. (Compare this with \cite[Definition 2.28]{pardon:contacthomology}.)

\subsection{Asymptotic operator in Levi-Civita connection}

Up until now, we have emphasized the usage of triad connection which
give rise to an optimal form of tensorial expression.
We recall that while $\nabla_Y J = 0$ for all $Y \in \xi$
for the contact triad connection $\nabla$ by
one of the defining axioms of the contact triad connection.
We have $\nabla_{R_\lambda}J \neq 0$ for the connection:
(In fact, $\nabla_{R_\lambda} J = 0$ if and only if
$R_\lambda$ is a Killing vector field, i.e.,
$\nabla R_\lambda = 0$ with respect to $\nabla$. See \cite[Remark 2.4]{oh-wang:CR-map1}.)
On the other hand while $\nabla^{\text{\rm LC}}_Y J \neq 0$
for the Levi-Civita connection $\nabla^{\text{\rm LC}}$ in general, the Levi-Civita
connection  carries  the following surprising property
$$
\nabla^{\text{\rm LC}}_{R_\lambda} J = 0.
$$
(See Lemma. 6.1 \cite{blair:book},
Proposition 4 \cite{oh-wang:connection}.)  Although it will not
be used in the present paper, we also convert the formula of the
asymptotic operator $A_{(\lambda,J,\nabla)}$ into one written
in terms of the Levi-Civita connection for a possible future purpose.

For this purpose, the following lemma is crucial.
\begin{lem}
For any $Y \in \Gamma(\xi)$, we have
\be\label{eq:LCR-triadR}
\nabla_{R_\lambda} Y = \nabla_{R_\lambda}^{\text{\rm LC}} Y
- \frac12 JY.
\ee
\end{lem}
\begin{proof} Applying the torsion property $T(R_\lambda,Y) = 0$ of the triad connection
$\nabla$ and the torsionfreeness of the Levi-Civita connection,
\eqref{eq:LCR-triadR} is equivalent to
$$
\nabla_Y R_\lambda  + [R_\lambda,Y] = \nabla^{\text{\rm LC}}_Y R_\lambda + [R_\lambda,Y]
- \frac12 JY.
$$
By cancelling the bracket terms away, it is enough to show
$$
\nabla_Y R_\lambda = \nabla^{\text{\rm LC}}_Y R_\lambda
- \frac12 JY.
$$
 This follows from  the formula  $\nabla_Y R_\lambda  = \frac12 (\CL_{R_\lambda}J) J$
(see \eqref{eq:nablaYRlambda}) and
$$
\nabla_Y^{\text{\rm LC}} R_\lambda
=  \frac12 JY + \frac12 (\CL_{R_\lambda}J) J
$$
for all $Y \in \xi$. (See Lemma 6.2 \cite{blair:book},  Lemma 9 \cite{oh-wang:connection} .)
\end{proof}
Therefore we have derived the simple relationship
$$
\nabla^{\text{\rm LC}}_{R_\lambda}
=  \frac12 J + \nabla_{R_\lambda}
$$
between $\nabla^{\text{\rm LC}}_{R_\lambda}$ and $\nabla_{R_\lambda}$
on $\Gamma(\xi)$.

Then taking the pull-back of this identity along the
loop $\gamma$ and combining Proposition \ref{prop:A} ,
we can rewrite the asymptotic operator in terms of
the Levi-Civita connection as follows.

\begin{cor}\label{cor:A-in-LC} Let $(\lambda, J)$ be any compatible pair
and let $\nabla^{\text{\rm LC}}$ be the Levi-Civita connection of the
triad metric of $(Q,\lambda,J)$. Let $w$ be a contact instanton with its asymptotic limit 
$\gamma = w(\pm \infty, \cdot)$  in cylindrical coordinate $(\tau,t)$ near a puncture.  Assume
$$
\int \gamma^*\lambda =T \neq 0
$$
and let $A_{(\lambda, J, \nabla)}$ be the asymptotic operator of $w$. Then we have
\begin{enumerate}
\item $[\nabla^{\text{\rm LC}}_t,J] (= \nabla^{\text{\rm LC}}_tJ)= 0$,
\item
\be
 \label{eq:A-in-LC}
A^\pi_{(\lambda,J,\nabla)} = - J\nabla_t^{\text{\rm LC}} - \frac{T}{2} Id +   \frac{T}{2} (\CL_{R_\lambda}J) J.
\ee
\end{enumerate}
\end{cor}

It follows from Corollary \ref{cor:A-in-LC} that the self-adjoint operator
$$
A^\pi_{(\lambda,J,\nabla)}: L^2(\gamma^*\xi) \to L^2(\gamma^*\xi)
$$
can be decomposed into
$A = A' + A''$ where $A'$ and $A''$ are the $J$-linear and
the anti-$J$-linear parts: We have
\be\label{eq:A'A''}
A' = -J \nabla_t^{\text{\rm LC}} - \frac{T}{2} Id,
\quad A'' = \frac{T}{2}\CL_{R_\lambda}J J
\ee
where $A''$ defines a compact operator on
$\text{\rm Dom}(A^\pi_{(\lambda,J,\nabla)})$ on a self-adjoint
extension of  $A^\pi_{(\lambda,J,\nabla)}$
which we regard as a linear map
$W^{1,2}(\gamma^*\xi) \to L^2(\gamma^*\xi)$.

The same kind of property also holds for the open string case
for the Legendrian pair $(R_0,R_1)$. The explicit formula for
the asymptotic operator given above enables us to study a  series of
perturbation results on the eigenfunctions and eigenvalues of
the asymptotic operators \emph{under the perturbation of $J$'s}
inside the set $\CJ_{\lambda}$ of $\lambda$-compatible CR almost
complex structures $J$. Discussion on this perturbation theory is now
in order.

\section{Spectral analysis of asymptotic operators}
\label{sec:spectral-perturbation}

In this section, we derive consequences of Corollary \ref{cor:A-in-LC} assuming
$T \neq 0$.

The ellipticity of $A^\pi_{(\lambda,J,\nabla)}$ implies that it
 has a discrete set of eigenvalues, which we enumerate into
$$
\cdots < \mu_{-k} < \cdots < \mu_{-1} < 0 < \mu_1 < \cdots < \mu_k <
\cdots
$$
with repeated finite multiplicity allowed,
where $\mu_k \to \infty$ (resp. $\mu_{-k} \to -\infty$) as $k \to \infty$.
It also has a uniform spectral gap (\cite[Theorem 6.29]{kato}) in that
there exists some $d = d(A^\pi_{(\lambda,J,\nabla)}) > 0$ such that
\be\label{eq:spectral-gap}
|\mu_{i+1} - \mu_i| \geq d > 0
\ee
for all $i$.

The explicit form of asymptotic operator on
$J$ and $\lambda$ given in the previous section
enables us to study the perturbation theory of
the asymptotic operator in terms of the change of $J$,
which we do in the next section.

In the rest of this section, we prove the following generic simpleness
statement using the perturbation theory of self-adjoint operators.

\begin{thm}[Simpleness of eigenvalues]
\label{thm:simple-eigenvalue}  Let $(Q,\xi)$ be a contact manifold.
Assume that $\lambda$ is
nondegenerate. For a generic choice of compatible pair $(\lambda,J)$,
 all eigenvalues $\mu_i$ of the asymptotic
operator are simple for all closed Reeb orbits.
\end{thm}

We first outline the scheme of the proof as follows:
\begin{enumerate}
\item For each $k$, we take the spectral decomposition
$$
W^{2,2}(\gamma^*\xi) = M_k \oplus M_k^\perp
$$
where $M_k: = \ker(A^\pi_{(\lambda,J,\nabla)} - \mu_k \text{\rm Id})$ is the eigenspace with
eigenvlue $\mu_i$.
We denote by
$\Pi_k:L^2(\gamma^*\xi) \to L^2(\gamma^*\xi)$ the  idempotent
associated to the orthogonal  projection
$\pi_k: L^2(\gamma^*\xi) \to M_k$.
We know that $M_k$ is a finite dimensional subspace and
$A^\pi_{(\lambda,J,\nabla)}$ restricts to a diagonalizable linear map denoted by $A_{k;\gamma}(J): M_k \to M_k$
for each $k = 0, \cdots $. We also define the associated map
\be\label{eq:Akgamma}
A_{k,\gamma}: \CJ_\lambda \to \End(M_k); \quad J \mapsto A_{k;\gamma}(J).
\ee
\item For each $k$, we apply the perturbation theory of
linear maps in finite dimensional vector spaces \cite[Chapter 1]{kato}
and show that the set of $J$ from which the
linear map $A_{k;\gamma}(J) $ has simple eigenvalues.
\end{enumerate}

\subsection{Spectral perturbation theory under $J$'s}

Let $k \geq 0$ be given and
$$
W^{2,2}(\gamma^*\xi) = M_k \oplus M_k^\perp
$$
be the decomposition mentioned as above.
We consider the diagonalizable linear map denoted by
$$
A_{k;\gamma}(J): M_k \to M_k
$$
given by restricting the operator $A^\pi_{(\lambda,J,\nabla)}$  to
$M_k$, i.e.,
$$
A_{k;\gamma} (J) := \pi_k A^\pi_{(\lambda,J,\nabla)}|_{M_k}
$$
where $\pi_k$ is the $L^2$-projection to $M_k$.

By choosing an $L^2$-orthonormal basis of $M_k$,
$$
\{e_1, e_2, \cdots, e_{m_k}\} = : \CB_k,
$$
we define a map
$$
\CJ_\lambda \to \End(\R^{m_k}) ; \quad J \mapsto
[A_{k;\gamma}(J)]_{\CB_k}
$$
where $[A_{k;\gamma}(J)]_{\CB_k}$ is the matrix of $A_{k;\gamma}(J)$ with respect to the basis $\CB_k$. This matrix is a symmetric matrix.

\begin{defn}\label{defn:discriminantal-variety} Define
$\Sym^{\text{\rm simp}}(\R^{m_k})$ to be the set of symmetric
matrices that has simple eigenvalues. We call the complement
\be\label{eq:discriminantal-variety}
\Sym(\R^{m_k}) \setminus \Sym^{\text{\rm simp}}(\R^{m_k})
\ee
the \emph{discriminantal variety}.
\end{defn}

Here the following remark provides that the complement
\eqref{eq:discriminantal-variety} is an algebraic variety
for which we can apply the (stratawise) transversality theorem.

\begin{rem}
Consider the characteristic polynomial of $A_{k;\gamma}(J)$ given by
$$
p_{k,\gamma}(J)(\mu): = \det (\mu \text{\rm Id} - A_{k;\gamma}(J))
$$
i.e.,
$$
p_{k,\gamma}(J) = \Char \circ A_{k;\gamma}(J)
$$
where $\Char(A): = \det((\cdot) I - A)$
is the characteristic
polynomial of the matrix $A$. We denote
its discriminant by $\Delta(p_{k,\gamma}(J))$: Recall that the discriminant
$\Delta(f) = \Delta_n(f)$ in general is a certain homogeneous polynomial with degree $2n-2$
of the coefficients of degree $n$ polynomial
$$
f = a_0 + a_1 x + \cdots + a_n x^n
$$
such that $f$ has a multiple root if and only if $\Delta(f) = 0$.
Therefore the zero set of $\Delta(f)$ defines a codimension 1 algebraic
variety of $\R P^{2n-2}$. (See \cite[Chapter 12]{gelfand-kap-zel} for
a summary of the properties of discriminants.) Then
\eqref{eq:discriminantal-variety} is the preimage of
$$
p_{k,\gamma}^{-1}(\Delta^{-1}(0)) = (\Delta \circ p_{k,\gamma})^{-1}(0).
$$
\end{rem}

Recall that the tangent space $T_J \CJ_\lambda$ can be written as
\bea\label{eq:TJCJ}
T_J \CJ_\lambda & = & \{ B \in \Gamma(\text{\rm End}(\xi)) \mid
 BJ + JB = 0,  \, d\lambda(B(\cdot), J (\cdot))
 + d\lambda(J(\cdot), B(\cdot)) =0 \}
 \nonumber\\
\eea
where the second equation means nothing but that $B$ is
a symmetric endomorphism of the metric $d\lambda(\cdot, J \cdot)|_\xi$
on $\xi$. (See \cite{floer:unregularized}, \cite[p.339]{oh:book1}
for a similar description in the symplectic case.)
We note that by definition $T_J \CJ_\lambda$ can be expressed as
 a fiber bundle
\be\label{eq:Slambda}
S_\lambda \to Q
\ee
whose fiber is given by $S_{\lambda,x} = T_J \CJ_\lambda|_x$
which is isomorphic to
$$
 S_{\lambda,x}= \{B \in \text{\rm End}(\R^{2n}) \mid
 B J_0 + J_0 B = 0,  \, d\lambda(B (\cdot),J_0(\cdot))
 + d\lambda(J_0 (\cdot), B(\cdot))=0 \}
 $$
 where $J_0$ is the standard complex structure of
 complex multiplication by $\sqrt{-1}$ with the identification
 $\R^{2n} \cong \C^n$.

Now we have the following proposition the proof of which
we postpone till the next subsection

\begin{prop}\label{prop:transversality} Let $\lambda$ be any
nondegenerate contact form $(Q,\xi)$, and let $(\gamma,T)$ be an
isospeed closed Reeb orbit with the loop $\gamma:S^1 \to Q$ and
$\int \gamma^*\lambda = T$. Consider the pull-back fiber bundle
$\gamma^*S_\lambda \to S^1$. Then the map
$$
A_{k;\gamma}: \CJ_\lambda \to \End(M_k) \cong M^{m_k\times m_k}(\R)
$$
is stratawise transverse to $(\Delta_{m_k})^{-1}(0)$
for all $\gamma \in \mathfrak{Reeb}(\lambda)$.
\end{prop}

Postponing the proof of this proposition till the next subsection,
we proceed with the proof of Theorem \ref{thm:simple-eigenvalue}.

This proposition proves the subset
$A_{k;\gamma}^{-1}(\Delta_{m_k}^{-1}(0))$
is a stratified smooth submanifold of real codimension 1 for each
given $k \geq 1$. In particular the complement
$$
\CJ_\lambda \setminus A_{k;\gamma}^{-1}(\Delta_{m_k}^{-1}(0))
$$
 is a residual subset of
$\CJ_\lambda$. Therefore their countable intersection
$$
\bigcap_{k \geq 1}
\left(\CJ_\lambda \setminus A_{k;\gamma}^{-1}(\Delta_{m_k}^{-1}(0)) \right)
= : \CJ_{\lambda}^{\text{\rm sm}}(\gamma)
$$
is still a residual subset of $\CJ_\lambda$. Since the set
$\mathfrak{Reeb}(\lambda)$ is a countable set, the intersection
$$
\bigcap_{\gamma \in \mathfrak{Reeb}(\lambda)}
\CJ_\lambda^{\text{\rm sm}}(\gamma)
$$
is still a residual subset of $\CJ_\lambda$. By definition, this last set
is precisely those $J$'s for which the associated asymptotic operator
carries simple eigenvalues.

This will complete the proof of
Theorem \ref{thm:simple-eigenvalue} except the proof of Proposition \ref{prop:transversality}
which is now in order.

\subsection{Proof of generic simpleness of eigenvalues}

Let $\gamma \in \mathfrak{Reeb}(\lambda)$ be given and $k$ be fixed.
We consider the assignment $J \mapsto A_{(\lambda, J, \nabla)}$
as a map
$$
F: \CJ_\lambda \to \Fred\left(W^{2,p}(\gamma^*\xi), W^{1,p}(\gamma^*\xi)\right),
$$
obtained by assigning the operator
$$
A^\pi_{(\lambda,J,\nabla)} = - J\nabla_t
+ \frac{T}{2} (\CL_{R_\lambda}J)J
$$
to $J$
in \eqref{eq:A-in-LC}
where $ \Fred(W^{2,2}(\gamma^*\xi), W^{1,2}(\gamma^*\xi))$
is the Banach space of Fredholm operators from $W^{2,2}(\gamma^*\xi)$
to $W^{1,2}(\gamma^*\xi)$.  To compute the variation of the assignment
$$
F: J \mapsto A_{(\lambda,J,\nabla)}^\pi
$$
we use \eqref{eq:nablaRlambdaY}
to convert the formula for $A^\pi_{(\lambda,J,\nabla)}$ into the following
which explicitly shows that the operator does not depend on the choice of
connection.
\begin{thm}\label{thm:A=T} Let $(Q,\lambda,J)$ be a triad. Let $(\gamma, T)$ be an
iso-speed Reeb trajectory. Then we have the formula
\be\label{eq:A-in-CL2}
A^\pi_{(\lambda,J,\nabla)} = T\left(-\frac12 \CL_{R_\lambda}J -  J \CL_{R_\lambda}
- \frac{1}{2} J(\CL_{R_\lambda}J)\right)
\ee
when acted upon $\Gamma(\gamma^*TM)$.
\end{thm}
\begin{proof} Recalling $\dot \gamma(t) = TR_\lambda(\gamma(t))$ and the definition of
the pull-back connection in general, we have
$$
(\nabla_t \eta)(t) = \nabla_{T R_\lambda(\gamma(t))} Y|_{\gamma(t)} =
T \nabla_{R_\lambda(\gamma(t))} Y|_{\gamma(t)}
$$
where $\eta \in \Gamma(\gamma^*\xi)$ and $Y$ is a  vector field $Y$
such that $\eta(t) = Y(\gamma(t))$ locally defined near
the point $\gamma(t)$ for given $t \in S^1$.

Therefore with a slight abuse of notation, utilizing \eqref{eq:nablaRlambdaY},
we rewrite
\beastar
A^\pi_{(\lambda,J,\nabla)}
& = & - J T \nabla_{R_\lambda} +   \frac{T}{2} (\CL_{R_\lambda}J) J \nonumber\\
& = & -\frac{T}2 J\CL_{R_\lambda}JJ - TJ \CL_{R_\lambda}
- \frac{T}{2} J(\CL_{R_\lambda}J)\nonumber\\
& = & -\frac{T}2 \CL_{R_\lambda}J - T J \CL_{R_\lambda}
- \frac{T}{2} J(\CL_{R_\lambda}J) \nonumber\\
& = & T\left(-\frac12 \CL_{R_\lambda}J -  J \CL_{R_\lambda}
- \frac{1}{2} J(\CL_{R_\lambda}J)\right).
\eeastar
This finishes the proof.
\end{proof}

Therefore its variation under $\delta J = B$ is given by
\bea\label{eq:dJFB}
d_J F(B)
 & = & T \left(- \frac12 \CL_{R_\lambda} B
- B \CL_{R_\lambda} -\frac12 B \CL_{R_\lambda}J
 - \frac12 J \CL_{R_\lambda} B\right) \nonumber \\
 & = & -\frac{T}2\left(\text{\rm Id} + J\right) \CL_{R_\lambda}
B - T B \left(\CL_{R_\lambda}
 + \frac12 \CL_{R_\lambda} J\right)
 \eea
for $B$ satisfying $JB + BJ = 0$.

In the following paragraph, we adapt the exposition of \cite[p.73]{ionel-parker:GW} in our
study of deformation of the asymptotic operators,
which is used in the context of Gromov-Witten theory.

We fix the Fredholm index $\iota$ and
We denote the set of such Fredholm operators (with index $\iota$) of $\dim_\R{\ker} = k$ by
$\mbox{\rm Fred}_k(J,\gamma;\mu)$ and their union by $\mbox{\rm Fred}_k^\iota(\gamma;\mu)$ i.e.,
\be\label{eq:Fred-J-mu}
\mbox{\rm Fred}_k^\iota(\gamma;\mu) = \bigcup_{J \in \CJ_\lambda}\mbox{\rm Fred}_k^\iota(J,\gamma;\mu)
\ee
as an infinite dimensional fiber bundle over $\CJ_\lambda$. In general,
by a theorem of Koschorke \cite{koschorke},
we have
$$
\text{\rm Fred}^\iota = \bigcup_{k,\gamma} \text{\rm Fred}_k^\iota
$$
where $\mbox{\rm Fred}_k^\iota$  of index $\iota$
is a submanifold of $\mbox{\rm Fred}$
with real codimension $k(k-\iota)$:
The normal bundle of $\mbox{\rm Fred}_k^\iota$ in $\mbox{\rm Fred}$ at an operator $D\in \mbox{Fred}_k$ is
$$
\mbox{Hom}_\R\left(\ker (D-\mu \text{\rm Id}),\mbox{coker} (D - \mu \text{\rm Id})\right).
$$
Considering the $L^2$-adjoint, we may identify this with
$$
\mbox{Hom}_\R(\ker(D - \mu \text{\rm Id}),\ker(D - \mu \text{\rm Id})).
$$

To show that a similar statement holds for the union of subsets thereof
\be\label{eq:Fred-J}
\mbox{\rm Fred}^\iota(\gamma;\mu) = \bigcup_{k}\mbox{\rm Fred}_k^\iota(J,\gamma;\mu).
\ee
we need to verify the restricted variations
arising from the change of $J$ is big enough.

\begin{prop}\label{prop:epimorphism} The map
$$
d_J A_{k;\gamma}: T_J\CJ_\lambda \to
T_{A_{k;\gamma}(J)}M^{m_k \times m_k}(\R) \cong
M^{m_k \times m_k}(\R)
$$
is an epimorphism at every $J \in (\Delta_{k;\gamma} \circ p_{k;\gamma})^{-1}(0)$.
\end{prop}
\begin{proof} To prove the proposition, we need to prove that for any element
$$
0 \neq \kappa, \, c \in \ker (A^\pi_{(\lambda,J,\nabla)} - \mu_k \text{\rm Id}) \subset \Gamma(\gamma^*\xi)
$$
we can find a variation $B \in T_J \CJ_\lambda$ such that
$$
\langle c, B \kappa \rangle_{L^2} \neq 0.
$$

The following lemma is a fundamental ingredient, the counterpart of \cite[Lemma C.1]{wendl:lecture}. However our variation is restricted to those arising
from $J$-variation of the assignment
$$
J \mapsto A^\pi_{(\lambda,J,\nabla)}
$$
based on the explicit $J$-dependent formula of $A^\pi_{(\lambda,J,\nabla)}$, while
Wendl used an abstract variation of Fredholm operators. Because of this
our proof is much more nontrivial than that of \cite[Lemma C.1]{wendl:lecture} which
strongly relies on the special form of $J$-dependence
of the linearization $ A^\pi_{(\lambda,J,\nabla)}$.

\begin{lem}\label{lem:wendl} Let $i \geq 0$ and
$L \in \End(\ker A^\pi_{(\lambda,J,\nabla)} - \mu_i \text{\rm Id})$ that is
symmetric with respect to the $L^2$ inner product. Then there exists a
smooth section
$$
B \in T_J \CJ_\lambda = \Gamma(\gamma^*S_\lambda) \subset \Gamma(\gamma^*\End(\xi))
$$
such that
$$
\langle c, B \kappa  \rangle_{L^2} = \langle c, L \kappa \rangle_{L^2}
$$
for all $\kappa, \, c \in \ker (A^\pi_{(\lambda,J,\nabla)} - \mu_i I)$.
\end{lem}
\begin{proof} We fix a basis
$
\CB: = \{e_1, \cdots, e_{m_i} \}
$
of $\ker (A^\pi_{(\lambda,J,\nabla)} - \mu_i \text{\rm Id})$.
Recall that each eigenfunction is a solution to
a first-order linear ODE and hence it is a nowhere varnishing
section of $\gamma^*\xi$. Furthermore
$$
\{e_1(t), \cdots, e_{m_i}(t) \} \in \xi_{\gamma(t)} = \CB(t)
$$
is linearly independent at each $t \in S^1$ and so spans
$m_i$ dimensional subspace of $\xi_{\gamma(t)}$.
In particular, we must have $m_i \leq 2n$ with $\dim Q = 2n+1$.

We extend the set $\CB(t)$ to an
orthogonal basis $\widetilde \CB(t)$ of $\xi_{\gamma(t)}$
and define a linear map $\widetilde L(t)$ on $\xi_{\gamma(t)}$ so that
\be\label{eq:tilde-L}
\widetilde L(t)(e_k(t)) = ( L e_k)(t)
\ee
for all $k = 1, \cdots, m_i$. This $\widetilde L(t)$ may neither be symmetric.
Since it satisfies
$\langle \kappa, \widetilde L c \rangle_{L^2} = \langle \kappa, L c \rangle_{L^2}$
for all $\kappa, \, c \in \ker(A^\pi_{(\lambda,J,\nabla)} - \mu_i \text{\rm Id})$,
we can replace $\widetilde L$ by its symmetrization $(\widetilde L + (\widetilde L)^T)/2$
if needed.

Writing $A = [dF(J)(B)]_\CB$, we have
$$
A_{ij} = \langle e_i, d_JF(B) e_j \rangle_{L^2}
=\int_{S^1} \left \langle e_i(t), (d_JF(B) e_j)(t)\right\rangle \, dt.
$$
\begin{lem}\label{lem:eigenfunction}
For any eigenfunction $\eta$ of $A_{(\lambda,J,\nabla)}$, we have
$$
\CL_{R_\lambda} \eta + \frac12 (\CL_{R_\lambda} J)\eta =  -\frac{\mu}{T} \eta -
\frac12 J (\CL_{R_\lambda}J) \eta.
$$
\end{lem}
\begin{proof} This is a direct consequence of \eqref{eq:A-in-CL2}.
\end{proof}

Therefore substituting this into \eqref{eq:dJFB}, we have
\beastar
d_JF(B) e_j & =  & -\frac{T}2\left(\text{\rm Id} + J\right) \CL_{R_\lambda}
Be_j  - T B \left(\CL_{R_\lambda} e_j
 + \frac12 (\CL_{R_\lambda} J) e_j \right) \\
 & = & \left( -\frac{T}2\left(\text{\rm Id} + J\right)( \CL_{R_\lambda}
B) + T B \left(\frac{\mu}{T}\, \text{\rm Id} +  \frac12 J(\CL_{R_\lambda}J)\right) \right) e_j.
\eeastar
We write the last operator in the big parenthesis as
$$
M \CL_{R_\lambda}B + B N
$$
where 
$$
M = -\frac{T}2\left(\text{\rm Id} + J\right), \quad 
N = - T(\frac{\mu}{T}\, \text{\rm Id} + \frac12 J \CL_{R_\lambda}J)
$$
and note that $M$ is an invertible map pointwise: The inverse of $M$ is given by
\be\label{eq:M-inverse}
M^{-1} = -\frac{1}{T}(\text{\rm Id} - J).
\ee
(Recall nondegeneracy of Reeb orbits implies $T \neq 0$.)

After pulling back the operator by $\gamma$,
we can write the pull-back operator as
$$
M \nabla^\phi_t B + BN
$$
and defining the map $\nabla^\phi_t$ by
$$
\nabla^\phi_t B := \CL_{R_\lambda} \widetilde B|_{\gamma(t)}
$$
where $\widetilde B$ is a smooth extension of the map
$$
\gamma(t) \mapsto B_t
$$
to a neighborhood of $\Image \gamma$. By the definition of the Lie derivative,
this definition does not depend on the choice of extension $\widetilde B$.
It is straightforward to check that this is the induced connection of
the one on the vector bundle $\gamma^*\xi \to S^1$ defined, which we still denote by
$\nabla^\phi_t$ as follows.
The following lemma is standard and so its proof is omitted.

\begin{lem} Let $(\gamma,T)$ be the given isospeed Reeb orbit.
Consider the pull-back vector bundle $\gamma^*\xi \to S^1$.
Define the operator $\nabla^\phi_t: \Gamma(\gamma^*\xi) \to
\Gamma(\gamma^*\xi)$ to be the one defined by
$$
\nabla^\phi_t(\eta): = T \CL_{R_\lambda}Y|_{\gamma(t)}
$$
for any (locally defined) vector field $Y$ tangent to $\xi$
defined near the point $\gamma(t)$
such that $Y(\gamma(t)) = \eta(t)$. Then the map is well-defined
and defines a connection on the vector bundle $\gamma^*\xi \to S^1$.
\end{lem}

Therefore we can express it as
$$
\nabla^\phi_t = \nabla_t + C_\phi
$$
where $\nabla_t$ is the pull-back connection of the Hermitian
 connection on $\xi \to Q$ induced by the contact triad connection
 on the triad $(Q,\lambda,J)$, and $C_\phi$ is a zero order operator
 along $\gamma$. Then we can rewrite $d_JF(B)$ into
$$
M \nabla_t B + M C_\phi B + BN.
$$
Furthermore we also have $\pi \nabla|_\xi = \pi \nabla^{\text{\rm LC}}|_\xi$.
(See \cite[Section 6]{oh-wang:connection}, more precisely see its arXiv version
1212.4817(v2).)

The following lemma is a crucial ingredient in our proof that enables us to
use a priori much smaller set of variations arising from $J$-variations
of $A^\pi_{(\lambda,J,\nabla)}$ instead of the abstract variations
used in \cite[Section 3.2]{wendl:lecture} especially in
the proof of \cite[Lemma C.1]{wendl:lecture}.
\begin{lem}\label{lem:solving-for-B}
The following initial valued problem can be uniquely solved:
$$
\begin{cases}
M \nabla^{\text{\rm LC}}_t B + M C_\phi B + BN = L \\
B(0) = B_0
\end{cases}
$$
for any given smooth section $L$ of $\End(\gamma^*\xi)$.
\end{lem}
\begin{proof} We first recall from \eqref{eq:M-inverse} that
the coefficient matrix $M$ is invertible. Therefore the equation can be rewritten as
$$
 \nabla^{\text{\rm LC}}_t B +  C_\phi B + M^{-1}BN = M^{-1} L
 $$
 which is a inhomogeneous linear first-order ODE for the variable $B$. By conjugating $B$ by the 
 parallel transport map $\Pi_\gamma^t: T_{\gamma(0)}Q \to T_{\gamma(t)}Q$, we consider
 $$
 \bar B(t): = \Pi_\gamma^t  B(t) (\Pi_\gamma^t)^{-1}: T_{\gamma(0)}Q \to T_{\gamma(0)}Q.
 $$
 Then the equation is further transformed to
 $$
\frac{d \bar B}{dt} +  \bar C_\phi \bar B 
+ \bar M^{-1}\bar B \bar N = \bar M^{-1} \bar L
$$
where all `bar' operators are conjugates by the parallel transport as for $\bar B$.
This is the genuine first-order linear system of ODE valued in the vector space $T_{\gamma(0)}Q$.
   
Now we apply the standard argument of `variation of constant' for 
solving the initial valued problem of the inhomogeneous first-order linear ODE
 for  $\bar B$. By writing back
$B(t) =  (\Pi_\gamma^t)^{-1} \bar B(t) \Pi_\gamma^t$, we have uniquely solved the required equation,
 which finishes the proof.
\end{proof}

Applying this lemma by setting $L  = \widetilde L$, the operator given in \eqref{eq:tilde-L}, 
we have finished the proof of Lemma \ref{lem:wendl}.
\end{proof}

This proves that the map $J \mapsto A^\pi_{(\lambda,J,\nabla)}$
is a submersion from $\CJ_\lambda$ to the stratum 
$\Delta_{k;\gamma}^{-1}(0)$ along
$$
(\Delta _{k;\gamma}\circ A_{k,\gamma})^{-1}(0)
$$
for all $k$, which finishes the proof of Proposition \ref{prop:epimorphism}.
\end{proof}

Now we have finished the proof of Proposition \ref{prop:transversality}.

\appendix

\section{Contact triad connection and canonical connection}
\label{sec:connection}

Let $(Q,\xi)$ be a given contact manifold. When a contact form $\lambda$
is given, we have the projection
$\pi=\pi_\lambda$ from $TQ$ to $\xi$ associated to the decomposition
$$
TQ = \xi \oplus \R \langle R_\lambda \rangle.
$$
We denote by $\Pi=\Pi_\lambda: TQ \to TQ$
 the corresponding idempotent, i.e., the endomorphism of $TQ$ satisfying
${\Pi}^2 = \Pi$, $\Im \Pi = \xi$, $\ker \Pi = \R\langle R_\lambda\rangle$.

\subsection{Contact triads and triad connections}

Let  $(Q, \lambda, J)$ be a contact triad of dimension $2n+1$ for the contact manifold $(Q, \xi)$, and equip with it the contact triad metric
$g=g_\xi+\lambda\otimes\lambda$.
In \cite{oh-wang:connection}, Wang and the second-named author
 introduced the \emph{contact triad connection} associated to every contact triad $(Q, \lambda, J)$ with the contact triad metric and proved its existence and uniqueness and naturality.

\begin{thm}[Contact Triad Connection \cite{oh-wang:connection}]\label{thm:connection}
For every contact triad $(Q,\lambda,J)$, there exists a unique affine connection $\nabla$, called the contact triad connection,
 satisfying the following properties:
\begin{enumerate}
\item The connection $\nabla$ is  metric with respect to the contact triad metric, i.e., $\nabla g=0$;
\item The torsion tensor $T$ of $\nabla$ satisfies $T(R_\lambda, \cdot)=0$;
\item The covariant derivatives satisfy $\nabla_{R_\lambda} R_\lambda = 0$, and $\nabla_Y R_\lambda\in \xi$ for any $Y\in \xi$;
\item The projection $\nabla^\pi := \pi \nabla|_\xi$ defines a Hermitian connection of the vector bundle
$\xi \to M$ with Hermitian structure $(d\lambda|_\xi, J)$;
\item The $\xi$-projection of the torsion $T$, denoted by $T^\pi: = \pi T$ satisfies the following property:
\be\label{eq:TJYYxi}
T^\pi(JY,Y) = 0
\ee
for all $Y$ tangent to $\xi$;
\item For $Y\in \xi$, we have the following
$$
\del^\nabla_Y R_\lambda:= \frac12(\nabla_Y R_\lambda- J\nabla_{JY} R_\lambda)=0.
$$
\end{enumerate}
\end{thm}
From this theorem, we see that the contact triad connection $\nabla$ canonically induces
a Hermitian connection $\nabla^\pi$ for the Hermitian vector bundle $(\xi, J, g_\xi)$,
and we call it the \emph{contact Hermitian connection}. This connection is
used for the study of various a priori estimates of contact instantons starting 
from \cite{oh-wang:CR-map1}.

\begin{cor}\label{cor:connection}
Let $\nabla$ be the contact triad connection. Then
\begin{enumerate}
\item For any vector field $Y$ on $Q$,
\be\label{eq:nablaYRlambda}
\nabla_Y R_\lambda = \frac{1}{2}(\CL_{R_\lambda}J)JY;
\ee
\item $\lambda(T)=d\lambda$ as a two-form on $Q$.
\end{enumerate}
\end{cor}

We refer readers to \cite{oh-wang:connection} for more discussion on the contact triad connection and its relation with other related canonical type connections.

\subsection{Almost Hermitian manifolds and canonical connections}

Next we give the definition of canonical connection on general almost
Hermitian manifolds and apply it to the case of $\mathfrak{lcs}$-fication of
contact triad $(Q,\lambda,J)$. See \cite[Chapter 7]{oh:book1} for the exposition of
canonical connection for almost Hermitian manifolds and its relationship with the Levi-Civita connection,
and in relation to the study of pseudoholomorphic curves on general symplectic manifolds
emphasizing the Weitzenb\"ock formulae in the study of elliptic regularity in the same spirit of the present survey.

Let $(M,J)$ be any almost complex manifold.
\begin{defn} A metric $g$ on $(M,J)$ is called Hermitian, if $g$ satisfies
$$
g(Ju,Jv) = g(u,v), \quad u, \, v \in T_x M, \, x \in M.
$$
We call the triple $(M,J,g)$ an almost Hermitian manifold.
\end{defn}

For any given almost Hermitian manifold $(M,J,g)$, the bilinear form
$$
\Phi : = g(J \cdot, \cdot)
$$
is called the fundamental two-form in \cite{kobayashi-nomizu2}, which
is nondegenerate.
\begin{defn} An almost Hermitian manifold $(M,J,g)$ is an \emph{almost
K\"ahler manifold} if the two-form $\Phi$ above is closed.
\end{defn}

\begin{defn}
A (almost) Hermitian connection $\nabla$ is an affine connection
satisfying
$$
\nabla g = 0 = \nabla J.
$$
\end{defn}
Existence of such a connection is easy to check.
In general the torsion $T = T_\nabla$ of the almost Hermitian connection $\nabla$ is not zero, even when
$J$ is integrable.  The following is the almost complex version of the Chern connection in complex
geometry.

\begin{thm}[\cite{gauduchon}, \cite{kobayashi}]
On any almost Hermitian manifold $(M,J,g)$,
there exists a unique Hermitian connection $\nabla$ on $TM$ satisfying
\be\label{eq:canonical-nabla}
T(X,JX) = 0
\ee
for all $X \in TM$.
\end{thm}
In complex geometry \cite{chern:connection} where $J$ is integrable, a Hermitian connection
satisfying \eqref{eq:canonical-nabla} is called the Chern connection.

\begin{defn}\label{defn:canonical-nabla} A \emph{canonical connection}
 of an almost Hermitian connection is
defined to be one that has the torsion property \eqref{eq:canonical-nabla}.
\end{defn}
The triple
$$
(Q \times \R, \widetilde J, \widetilde g_\lambda)
$$
is a natural example of an almost Hermitian manifold associated to the
contact triad $(Q,\lambda,J)$.

Let $\widetilde \nabla$ be the canonical
connection thereof. Then we have the following which also provides a natural
relationship between the contact triad connection and the canonical connection.

\begin{prop}[Canonical connection versus contact triad connection]\label{prop:canonical}
Let $\widetilde g = \widetilde g_\lambda $ be the almost Hermitan metric given above.
Let $\widetilde \nabla$ be the canonical connection of the almost Hermitian manifold
\eqref{eq:almost-Hermitian-triple}, and $\nabla$ be the contact triad connection for the triad $(Q,\lambda,J)$.
Then $\widetilde \nabla$ preserves the splitting \eqref{eq:TM-splitting}
and satisfies $\widetilde \nabla|_\xi = \nabla|_\xi$.
\end{prop}

\section{Covariant differential of vector-valued forms}\label{sec:covariant-differential}

In this appendix, we recall the standard exterior calculus of vector-valued forms
borrowing from the exposition in \cite[Appendix]{oh-wang:CR-map1}.

Assume $(P, h)$ is a Riemannian manifold of dimension $n$ with metric $h$, and $D$ is the Levi--Civita connection.
Let $E\to P$ be any vector bundle with inner product $\langle\cdot, \cdot\rangle$,
and assume $\nabla$ is a connection on $E$ which is compatible with $\langle\cdot, \cdot\rangle$.

Denote by $\Omega^k(E)$ the space of $E$-valued $k$-forms on $P$. The connection $\nabla$
induces an exterior derivative by
\beastar
d^\nabla&:& \Omega^k(E)\to \Omega^{k+1}(E)\\
d^\nabla(\alpha\otimes \zeta)&=&d\alpha\otimes \zeta+(-1)^k\alpha\wedge \nabla\zeta.
\eeastar
It is not hard to check that for any $1$-forms $\beta$, equivalently one can write
$$
d^\nabla\beta (v_1, v_2)=(\nabla_{v_1}\beta)(v_2)-(\nabla_{v_2}\beta)(v_1),
$$
where $v_1, v_2\in TP$. 

For the purpose of the present paper, we mostly apply the last formula to the $w^*TQ$ 
or $w^*\xi$ one-forms. Some illustrations thereof are now in order.

For any given map $w: \dot \Sigma \to Q$, not necessarily arising from the symplectization,
we can decompose its derivative
$dw$, regarded as a $w^*TQ$-valued one-form on $\dot \Sigma$, into
\be\label{eq:du}
dw = d^\pi w + w^*\lambda \otimes R_\lambda
\ee
 where $d^\pi w := \pi dw$. Furthermore $d^\pi w$ is decomposed into
\be\label{eq:dpiu}
d^\pi w = \delbar^\pi w + \del^\pi w
\ee
where $\delbar^\pi w: = (dw^\pi)_J^{(0,1)}$ (resp. $\del^\pi w: = (dw^\pi)_J^{(1,0)}$) is
the anti-complex linear part (resp. the complex linear part) of $d^\pi w: (T\dot \Sigma, j) \to (\xi,J|_\xi)$.
(For the simplicity of notation, we will abuse our notation by often denoting $J|_\xi$ by $J$.
We also simply write $((\cdot)^\pi)_J^{(0,1)} = (\cdot)^{\pi(0,1)}$ and
$((\cdot)^\pi)_J^{(1,0)} = (\cdot)^{\pi(1,0)}$ in general.)

Here are some examples of vector-valued forms that appear in the main text of the present paper:
\begin{enumerate}
\item The vector fields $Y$ (resp. $Y^\pi$) along the map $w$ are 
$w^*TQ$-valued (resp. $w^*\xi$-valued) zero-form.
\item The covariant derivative $\nabla Y$ is a $w^*TQ$-valued one-form.
\item $\nabla^\pi Y^\pi$ is a $w^*\xi$-valued one-form.
\item We regard $dw$ (resp. $d^\pi w$) as $w^*TQ$-valued (resp. $w^*\xi$-valued) one-form.
\end{enumerate}

\section{Subsequence convergence and charge vanishing}
\label{sec:charge-vanishing}

In this section, we recall subsequence and charge vanishing result on contact instantons from
\cite{oh-wang:CR-map1}, \cite{oh-yso:index}.

We put the following hypotheses in our asymptotic study of the finite
energy contact instanton maps $w$ as in \cite{oh-wang:CR-map1}:
\begin{hypo}\label{hypo:basic-intro}
Let $h$ be the metric on $\dot \Sigma$ given above.
Assume $w:\dot\Sigma\to M$ satisfies the contact instanton equations \eqref{eq:contacton-Legendrian-bdy-intro},
and
\begin{enumerate}
\item $E^\pi_{(\lambda,J;\dot\Sigma,h)}(w)<\infty$ (finite $\pi$-energy);
\item $\|d w\|_{C^0(\dot\Sigma)} <\infty$.
\item $\Image w \subset \mathsf K\subset M$ for some compact subset $\mathsf K$.
\end{enumerate}
\end{hypo}

The above finite $\pi$-energy and $C^0$ bound hypotheses imply
\be\label{eq:hypo-basic-pt}
\int_{[0, \infty)\times [0,1]}|d^\pi w|^2 \, d\tau \, dt <\infty, \quad \|d w\|_{C^0([0, \infty)\times [0,1])}<\infty
\ee
in these coordinates.

\begin{defn}[Asymptotic action and charge]
Assume that the limit of $w(\tau, \cdot)$ as $\tau \to \infty$ exists.
Then we can associate two
natural asymptotic invariants at each puncture defined as
\bea
T & := & \lim_{r \to \infty} \int_{[0,1]} w_r^*\lambda \label{eq:TQ-T}\\
- Q & : = & \lim_{r \to \infty} \int_{[0,1]} w_r^*(\lambda \circ j).\label{eq:TQ-Q}
\eea
where $w_r: [0,1]\to Q$ is the map defined by $w_r(t) := w(r,t)$.
(Here we only look at positive punctures. The case of negative punctures is similar.)
We call $T$ the \emph{asymptotic contact action}
and $Q$ the \emph{asymptotic contact charge} of the contact instanton $w$ at the given puncture.
\end{defn}

The following open string version
of subsequence convergence and charge vanishing result
is proved in \cite{oh:contacton-Legendrian-bdy}, \cite{oh-yso:index}.

\begin{thm}\label{thm:subsequence-open}
Let $w:[0, \infty)\times [0,1]\to M$ satisfy the contact instanton equations \eqref{eq:contacton-Legendrian-bdy-intro}
and converges $|\tau| \to \infty$.
Then for any sequence $s_k\to \infty$, there exists a subsequence, still denoted by $s_k$, and a
massless instanton $w_\infty(\tau,t)$ (i.e., $E^\pi(w_\infty) = 0$)
on the strip $\R \times [0,1]$  that satisfies the following:
\begin{enumerate}
\item $\delbar^\pi w_\infty = 0$ and
$$
\lim_{k\to \infty}w(s_k + \tau, t) = w_\infty(\tau,t)
$$
in the $C^l(K \times [0,1], M)$ sense for any $l$, where $K\subset [0,\infty)$ is an arbitrary compact set.
\item $w_\infty$ has vanishing asymptotic charge $Q = 0$ and satisfies $w_\infty(\tau,t)= \gamma(t)$
for some Reeb chord $\gamma$ is some Reeb chord joining $R_0$ and $R_1$ with period $T$ at each puncture.
\item $T \neq 0$ at each  puncture with the associated pair $(R,R')$ of boundary condition with $R \cap R'
= \emptyset$.
\end{enumerate}
\end{thm}

\begin{cor}[Corollary 5.11 \cite{oh-yso:index}]
\label{cor:tangent-convergence}
Assume that the pair $(\lambda, \vec R)$ is nondegenerate
in the sense of Definition \ref{defn:nondegeneracy-links}.
Let $w: \dot \Sigma \to M$ satisfy the contact instanton equation \eqref{eq:contacton-Legendrian-bdy-intro} and Hypothesis \eqref{eq:hypo-basic-pt}.
Then on each strip-like end with strip-like coordinates $(\tau,t) \in [0,\infty) \times [0,1]$ near a puncture
\beastar
&&\lim_{s\to \infty}\left|\pi \frac{\del w}{\del\tau}(s+\tau, t)\right|=0, \quad
\lim_{s\to \infty}\left|\pi \frac{\del w}{\del t}(s+\tau, t)\right|=0\\
&&\lim_{s\to \infty}\lambda(\frac{\del w}{\del\tau})(s+\tau, t)=0, \quad
\lim_{s\to \infty}\lambda(\frac{\del w}{\del t})(s+\tau, t)=T
\eeastar
and
$$
\lim_{s\to \infty}|\nabla^l dw(s+\tau, t)|=0 \quad \text{for any}\quad l\geq 1.
$$
All the limits are uniform for $(\tau, t)$ in $K\times [0,1]$ with compact $K\subset \R$.
\end{cor}

We also state that the same holds for closed string case for which
the charge $Q$ may not vanish.
The proof of the following subsequence convergence result
is proved in \cite[Theorem 6.4]{oh-wang:CR-map1}.
\begin{thm}[Subsequence Convergence,
Theorem 6.4 \cite{oh-wang:CR-map1}]
\label{thm:subsequence}
Let $w:[0, \infty)\times S^1 \to M$ satisfy the contact instanton 
equation \eqref{eq:contacton-Legendrian-bdy-intro}
and Hypothesis \eqref{eq:hypo-basic-pt}.
Then for any sequence $s_k\to \infty$, there exists a subsequence, still denoted by $s_k$, and a
massless instanton $w_\infty(\tau,t)$ (i.e., $E^\pi(w_\infty) = 0$)
on the cylinder $\R \times [0,1]$  that satisfies the following:
\begin{enumerate}
\item $\delbar^\pi w_\infty = 0$ and
$$
\lim_{k\to \infty}w(s_k + \tau, t) = w_\infty(\tau,t)
$$
in the $C^l(K \times [0,1], M)$ sense for any $l$, where $K\subset [0,\infty)$ is an arbitrary compact set.
\item $w_\infty^*\lambda = -Q\, d\tau + T\, dt$
\end{enumerate}
\end{thm}

We have the same exponential convergence as Corollary
\ref{cor:tangent-convergence} for the closed strong case, provided $Q = 0$.

\begin{cor}\label{cor:Q=0} Assume that $\lambda$ is nondegenerate.
Suppose that $w_\tau$ converges as $|\tau| \to \infty$  and its massless
limit instanton has  $Q = 0$ but $T \neq 0$, then the $w_\tau$
converges to a Reeb orbit of period $|T|$ exponentially fast.
\end{cor}

\section{Off-shell setting of the linearization operator}
\label{sec:off-shell}

We now provide some details of the Fredholm theory and the index calculation
referring readers to the original articles \cite{oh:contacton}, \cite{oh-savelyev} for 
complete details.

We fix an elongation function $\rho: \R \to [0,1]$
so that
\beastar
\rho(\tau) & = & \begin{cases} 1 \quad & \tau \geq 1 \\
0 \quad & \tau \leq 0
\end{cases} \\
0 & \leq & \rho'(\tau) \leq 2.
\eeastar
Then we consider sections of $w^*TQ$ by
\be\label{eq:barY}
\overline Y_i = \rho(\tau - R_0) R_\lambda(\gamma^+_i(t)),\quad
\underline Y_j = \rho(\tau + R_0) R_\lambda(\gamma^+_j(t))
\ee
and denote by $\Gamma_{s^+,s^-} \subset \Gamma(w^*TQ)$ the subspace defined by
$$
\Gamma_{s^+,s^-} = \bigoplus_{i=1}^{s^+} \R\{\overline Y_i\} \oplus \bigoplus_{j=1}^{s^-} \R\{\underline Y_j\}.
$$
Let $k \geq 2$ and $p > 2$.
The local model of the tangent space  of $\CW^{k,p}_\delta(\dot \Sigma, Q;J;\gamma^+,\gamma^-)$ at
$$
w \in C^\infty_\delta(\dot \Sigma,Q) \subset W^{k,p}_\delta(\dot \Sigma, Q)
$$
is given by
\be\label{eq:tangentspace}
\Gamma_{s^+,s^-} \oplus W^{k,p}_\delta(w^*TQ)
\ee
where $W^{k,p}_\delta(w^*TQ)$ is the Banach space
\beastar
&{} & \{Y = (Y^\pi, \lambda(Y)\, R_\lambda)
\mid e^{\frac{\delta}{p}|\tau|}Y^\pi \in W^{k,p}(\dot\Sigma, w^*\xi), \,
\lambda(Y) \in W^{k,p}(\dot \Sigma, \R)\}\\
& \cong & W^{k,p}(\dot \Sigma, \R) \cdot R_\lambda(w) \oplus W^{k,p}(\dot\Sigma, w^*\xi).
\eeastar
Here we measure the various norms in terms of the triad metric of the triad $(Q,\lambda,J)$.

We choose $\delta> 0$ so that $0 < \delta/p < 1$ is smaller than the
spectral gap
\be\label{eq:gap}
\text{gap}(\gamma^+,\gamma^-): = \min_{i,j}
\{d_{\text H}(\text{spec}A_{(T_i,z_i)},0),\, d_{\text H}(\text{spec}A_{(T_j,z_j)},0)\}.
\ee
 We denote by
$$
\CW^{k,p}_\delta(\dot \Sigma, Q;J;\gamma^+,\gamma^-), \quad k \geq 2
$$
the Banach manifold such that
$$
\lim_{\tau \to \infty}w((\tau,t)_i) = \gamma^+_i(T_i(t+t_i)), \quad
\lim_{\tau \to - \infty}w((\tau,t)_j) = \gamma^-_j(T_j(t-t_j))
$$
for some elements $t_i, \, t_j \in S^1$, where $(\tau,t)_i$ and $(\tau,t)_j$ are the cylindrical
coordinates at the punctures $i, \, j$ respectively and
$$
T_i = \int_{S^1} (\gamma^+_i)^*\lambda, \, T_j = \int_{S^1} (\gamma^-_j)^*\lambda.
$$
Here $t_i$ (resp. $t_j$) depends on the given analytic coordinate $(\tau,t)_i$
(resp. $(\tau,t)_j$) and the parametrization of
the Reeb orbits.

Now for each given $w \in \CW^{k,p}_\delta:= \CW^{k,p}_\delta(\dot \Sigma, Q;J;\gamma^+,\gamma^-)$,
we consider the Banach space
$$
\Omega^{(0,1)}_{k-1,p;\delta}(w^*\xi)
$$
the $W^{k-1,p}_\delta$-completion of $\Omega^{(0,1)}(w^*\xi)$ and form the bundle
$$
\CH^{(0,1)}_{k-1,p;\delta}(\xi) = \bigcup_{w \in \CW^{k,p}_\delta} \Omega^{(0,1)}_{k-1,p;\delta}(w^*\xi)
$$
over $\CW^{k,p}_\delta$. Then we can regard the assignment
$$
\Upsilon_1: (w,f) \mapsto \delbar^\pi w
$$
as a smooth section of the bundle $\CH^{(0,1)}_{k-1,p;\delta}(\xi) \to \CW^{k,p}_\delta$.

Similarly we consider
$$
\Omega^{(0,1)}_{k-1,p;\delta}(u^*\CV)
$$
the $W^{k-1,p}_\delta$-completion of $\Omega^{(0,1)}(u^*\CV)$ and form the bundle
$$
\CH^{(0,1)}_{k-1,p;\delta}(\CV) = \bigcup_{w \in \CW^{k,p}_\delta} \Omega^{(0,1)}_{k-1,p;\delta}(u^*\CV)
$$
over $\CW^{k,p}_\delta$. Then the assignment
$$
\Upsilon_2: u=(w,f) \mapsto w^*\lambda \circ j - f^*ds
$$
defines a smooth section of the bundle
$$
\CH^{(0,1)}_{k-1,p;\delta}(\CV) \to \CW^{k,p}_\delta.
$$
We have already computed the linearization of each of these maps in the previous section.

With these preparations, the following is a corollary of exponential estimates established
in \cite{oh-wang:CR-map1}.

\begin{prop}[Corollary 6.5 \cite{oh-wang:CR-map1}]\label{prop:on-containedin-off}
Assume $\lambda$ is nondegenerate.
Let $w:\dot \Sigma \to Q$ be a contact instanton and let $w^*\lambda = a_1\, d\tau + a_2\, dt$.
Suppose
\bea
\lim_{\tau \to \infty} a_{1,i} = -Q(p_i), &{}& \, \lim_{\tau \to \infty} a_{2,i} = T(p_i)\nonumber\\
\lim_{\tau \to -\infty} a_{1,j} = -Q(p_j) , &{}& \, \lim_{\tau \to -\infty} a_{2,j} = T(p_j)
\eea
at each puncture $p_i$ and $q_j$.
Then $w \in \CW^{k,p}_\delta(\dot \Sigma, Q;J;\gamma^+,\gamma^-)$.
\end{prop}

Now we are ready to describe the moduli space of pseudoholomorphic curves on symplectization
with prescribed asymptotic condition as the zero set
\be\label{eq:defn-MM}
\CM(\dot \Sigma, Q;J;\gamma^+,\gamma^-) = \left(\CW^{k,p}_\delta(\dot \Sigma, Q;J;\gamma^+,\gamma^-)
\oplus \CW^{k,p}_\delta(\dot \Sigma, \R)\right) \cap \Upsilon^{-1}(0)
\ee
whose definition does not depend on the choice of $k, \, p$ or $\delta$ as long as $k\geq 2, \, p>2$ and
$\delta > 0$ is sufficiently small. One can also vary $\lambda$ and $J$ and define the universal
moduli space whose detailed discussion is postponed.
We see therefrom that
$D\Upsilon_{(\lambda,J)}$ is the first-order differential operator whose first-order part
is given by the direct sum operator
$$
(Y^\pi,(\kappa, b)) \mapsto \delbar^{\nabla^\pi} Y^\pi \oplus (d\kappa \circ j - db)
$$
where we write $(Y,v) = \left(Y^\pi + \kappa R_\lambda, b \frac{\del}{\del s}\right)$
for $\kappa = \lambda(Y), \, b = ds(v)$.
Here we have
$$
\delbar^{\nabla^\pi} : \Omega^0_{k,p;\delta}(w^*\xi;J;\gamma^+,\gamma^-) \to
\Omega^{(0,1)}_{k-1,p;\delta}(w^*\xi)
$$
and the second summand can be written as the standard Cauchy-Riemann operator
$$
\delbar: W^{k,p}(\dot \Sigma;\C) \to \Omega^{(0,1)}_{k-1,p}(\dot \Sigma,\C); \quad
b + i \kappa =: \varphi \mapsto
\delbar \varphi.
$$
The following proposition can be derived from the arguments used by
Lockhart and McOwen \cite{lockhart-mcowen}. 

\begin{prop}\label{prop:fredholm}
Suppose $\delta > 0$ satisfies the inequality
$$
0< \delta < \min\left\{\frac{\text{\rm gap}(\gamma^+,\gamma^-)}{p}, \frac{2}{p}\right\}
$$
where $\text{\rm gap}(\gamma^+,\gamma^-)$ is the spectral gap, given in \eqref{eq:gap},
of the asymptotic operators $A_{(T_j,z_j)}$ or $A_{(T_i,z_i)}$
associated to the corresponding punctures. Then the operator
\eqref{eq:DUpsilonu} is Fredholm.
\end{prop}

\def\cprime{$'$}
\providecommand{\bysame}{\leavevmode\hbox to3em{\hrulefill}\thinspace}
\providecommand{\MR}{\relax\ifhmode\unskip\space\fi MR }
\providecommand{\MRhref}[2]{%
  \href{http://www.ams.org/mathscinet-getitem?mr=#1}{#2}
}
\providecommand{\href}[2]{#2}


\end{document}